\documentclass[11pt]{amsart}
\usepackage{geometry}                
\usepackage[parfill]{parskip}    
\usepackage{graphicx}
\usepackage{amssymb}
\usepackage{epstopdf}
\newtheorem{theorem}{Theorem}[section]

\newtheorem{lemma}[theorem]{Lemma}
\newtheorem{proposition}{Proposition}

\theoremstyle{definition}
\newtheorem{definition}[theorem]{Definition}
\newtheorem{remark}{Remark}

\newcommand{\R}{\mathbb{R}}
\newcommand{\N}{\mathbb{N}}

\title[Global Dynamics]
      {Global dynamics of planar discrete type-K competitive systems}

\author[Zhanyuan Hou]{}


 \email{z.hou@londonmet.ac.uk}

\begin{document}

\maketitle

\centerline{\scshape Zhanyuan Hou }
\medskip
{\footnotesize
 \centerline{School of Computing and Digital Media, London Metropolitan University,}
   \centerline{166-220 Holloway Road, London N7~8DB, UK}
} 

\medskip


\begin{abstract}
For a continuously differentiable Kolmogorov map defined from the nonnegative orthant to itself, a type-K competitive system is defined. Under the assumptions that the system is dissipative and the origin is a repeller, the global dynamics of such systems is investigated. A (weakly) type-K retrotone map is defined on a bounded set, which is backward monotone in some order. Under certain conditions, the dynamics of a type-K competitive system is the dynamics of a type-K retrotone map. Under these conditions, there exists a compact invariant set \textbf{A} that is the global attractor of the system on the nonnegative orthant exluding the origin. Some basic properties of \textbf{A} are established and remaining problems are listed for further investigation for general $N$-dimensional systems. These problems are completely solved for planar type-K competitive systems: every forward orbit is eventually monotone and converges to a fixed point; the global attractor \textbf{A} consists of two monotone curves each of which is a one-dimensional compact invariant manifold. A concrete example is provided to demonstrate the results for planar systems. 
\end{abstract}

\textbf{Note.} The final version of this paper will be published in Journal of Difference Equations and Applications soon.


\section{Introduction}
Consider the discrete dynamical system
\begin{equation}\label{e1}
	x(n) = T^n(x), \quad x\in C,\quad n\in \N,
\end{equation}
where $C = \R^N_+ = [0, +\infty)^N$, $\N = \{0, 1, 2, \ldots\}$ and the map $T: C \to C$ has the form
\begin{equation}\label{e2}
	T_i(x) = x_if_i(x), \quad i\in I_N = \{1, 2, \ldots, N\}
\end{equation}
and $f\in C^1(C, C)$ with $f_i(x)>0$ for all $x\in C$ and $i\in I_N$. System (\ref{e1}) is often viewed in literature as a mathematical model for the population dynamics of a community of $N$ species, where each $x_i(n)$ represents the population size or density at time $n$ (at the end of $n$th time period), and the function $f_i(x)$ denotes the per capita growth rate, of the $i$th species. 

If increase of the $j$th population reduces the per capita growth rate of the $i$th species, so $\frac{\partial f_i}{\partial x_j} \leq 0$, for all $i, j\in I_N$, then (\ref{e1}) models the population dynamics of a community of competitive species and (\ref{e1}) is called a {\it competitive system}. If increase of the $j$th population promotes the per capita growth rate of the $i$th species, so $\frac{\partial f_i}{\partial x_j} \geq 0$, for all $i, j\in I_N$, then (\ref{e1}) models the population dynamics of a community of cooperative species and (\ref{e1}) can be called a {\it cooperative system}. (Note that these concepts are different from cooperative and competitive maps appeared in literature, e.g. in \cite{Smi}, \cite{WaJi1} or \cite{Kul}.) Now suppose the $N$-species community is divided into two groups, $(x_1, \ldots, x_k)$ in group 1 and $(x_{k+1}, \ldots, x_N)$ in group 2, species in each group form a competitive system but species between the two groups are cooperative. This situation can be described by the signs of $\frac{\partial f_i}{\partial x_j}$ in the Jacobian matrix $Df(x) = (\frac{\partial f_i}{\partial x_j})$, 
\begin{equation}\label{e3}
	Df(x) = \left(\begin{array}{ll} A_1 & B_1 \\ B_2 & A_2 \end{array} \right),
\end{equation}
where $A_1$ is a $k\times k$ matrix, $A_2$ is $(N-k)\times(N-k)$ matrix, of nonpositive entries with each diagonal entry negative, and $B_1$ and $B_2$ are matrices with nonnegative entries. Then (\ref{e1}) models the population dynamics of $N$ species cooperating between two competitive groups. In this case, (\ref{e1}) is called a {\it type-K competitive system}. 

For cooperative system (\ref{e1}) $T$ is actually a monotone map. So the general theory of monotone dynamical systems (see \cite{Smi1}, \cite{HirSmi} and the references therein) can be applied to (\ref{e1}) for its global dynamics. For competitive systems, (\ref{e1}) and its various particular instances as models have attracted huge interests from researchers in the last few decades. One of the important and influential developments is the existence of a {\it carrying simplex} $\Sigma \subset C$ when $T$ restricted to a compact set $S\subset C$ is invertible and $T(S)\subset S$. The set $\Sigma$ is actually the global attractor of the system on $C\setminus\{\mathbf{0}\}$. For more descriptions of $\Sigma$ and the global dynamics in terms of $\Sigma$, see \cite{Hir1}, \cite{Hou1}, \cite{WaJi1}, \cite{WaJi2}, \cite {JiMiWa}, \cite{Her}, \cite{Hou2} and \cite{MieBai}. 

For type-K competitive differential systems, the author's recent work \cite{Hou3} based on the results of \cite{LiaJia} investigates the global attractor of the system restricted to $C\setminus\{\mathbf{0}\}$ and describes the global dynamics in terms of this global attractor, which can be viewed as an analogue to the carrying simplex of a competitive system. For discrete type-K competitive system (\ref{e1}), however, little is known about the global dynamics of the system within the author's knowledge. 

The aim of this paper is to start the investigation on the global dynamics of type-K competitive system (\ref{e1}). We shall further describe the problems for $N$-dimensional systems in section 2, and focus on a thorough exploration for planar type-K competitive systems (\ref{e1}) in section 3. The the rest of the paper is for the proof of some results and conclusion.

\section{Type-K retrotone maps and preliminaries} 

We first introduce some notation. Let $k\in I_N$ with $k<N$ and 
\[
K= \{p\in\R^N: \forall i\in I_k, p_i\geq 0; \forall j\in I_N\setminus I_k, p_j\leq 0\}
\] 
with $\dot{K}$ $(\dot{C})$ denoting the interior of $K$ $(C)$. For any $x, y\in \R^N$, we write $x\leq y$ or $y \geq x$ if $y-x\in C$; $x<y$ or $y>x$ if $x\leq y$ but $x\not= y$; $x\ll y$ or $y\gg x$ if $y-x \in \dot{C}$. Similarly, $x\leq_K y$ or $y \geq_K x$ if $y-x\in K$; $x<_K y$ or $y>_K x$ if $x\leq_K y$ but $x\not= y$; $x\ll_K y$ or $y\gg_K x$ if $y-x \in \dot{K}$. For any $x, y\in C$ with $x\leq y$ ($x\leq_K y$), let $[x, y] = \{z\in C: x\leq z\leq y\}$ ($[x, y]_K = \{z\in C: x\leq_K z\leq_K y\}$). When $[$ or $]$ is replaced by $($ or $)$ in the above notation, the corresponding $\leq$ ($\leq_K$) is replaced by $\ll$ ($\ll_K$). For any $i\in I_N$, let $\pi_i = \{x\in C: x_i = 0\}$, the $i$th coordinate plane in $C$. For any $J\subset I_N$ we define $C_J$ by
\[
C_J = \cap_{i\in I_N\setminus J}\pi_i = \{x\in C: \forall i\in I_N\setminus J, x_i=0\}.
\]
Then, with $H = I_k =\{1, \ldots, k\}$ and $V = I_N\setminus H$, $C$ can be viewed as  ``two-dimensional'' with ``horizontal positive half axis'' $C_H$ and ``vertical positive half axis'' $C_V$.

\begin{definition}
	The map $T$ given by (\ref{e2}) is said to be {\it type-K retrotone} on a subset $X\subset C$ if for any $x, y\in X$, $T(x) <_K T(y)$ implies $x <_K y$ and, for each $i\in H$ and $j\in V$, if $y_i\not= 0$ then $x_i<y_i$ and if $x_j\not= 0$ then $x_j>y_j$. 
\end{definition}

\begin{definition}
	The map $T$ given by (\ref{e2}) is said to be {\it weakly type-K retrotone} on a subset $X\subset C$ if for any $x, y\in X$, $T(x) <_K T(y)$ implies $x <_K y$ and, for each $i\in H$ and $j\in V$, if $T_i(x)<T_i(y)$ then $x_i<y_i$ and if $T_j(x)>T_j(y)$ then $x_j>y_j$. 
\end{definition}

For $x, y\in X$, let $I_1 = \{i\in I_N: \max\{x_i, y_i\}>0\}$ and $J_1= \{j\in I_N: T_j(x)\not= T_j(y)\}$. Then $J_1\subset I_1$. If $T$ is a type-K retrotone map on $X$, then $T(x) <_K T(y)$ implies $x <_K y$ and, restricted to $C_{I_1}$, $x\ll_K y$. But if $T$ is a weakly type-K retrotone map on $X$, then $T(x) <_K T(y)$ implies $x <_K y$ and, restricted to $C_{J_1}$, $x\ll_K y$. Since $J_1$ is a subset of $I_1$, a type-K retrotone map is obviously weakly type-K retrotone but not vice versa. Note that a (weakly) type-K retrotone map restricted to $X\cap C_V$ or $X\cap C_H$ is (weakly) retrotone (see \cite{Her}, \cite{Hou2} or \cite{MieBai} for definition).

In parallel to retotone maps related to competitive systems, we shall see that type-K retotone maps are related to type-K competitive systems. From now on we assume the following conditions for system (\ref{e1}):
\begin{itemize}
	\item[(A1)] (\ref{e1}) is a type-K competitive system.
	\item[(A2)] $f(\mathbf{0})\gg \mathbf{1}=(1, \ldots, 1)$ so that the origin $\mathbf{0}=(0, \ldots, 0)$ is a repeller.
	\item[(A3)] (\ref{e1}) is dissipative.
\end{itemize}
Then, from (A3) we see the existence of a bounded set $B\subset C$ such that $T^n(x)\in B$ for each $x\in C$ and large enough $n$. So it is reasonable to assume the existence of $r\in \dot{C}$ such that $T$ maps $[\mathbf{0}, r]$ into itself: $T([\mathbf{0}, r]) \subset [\mathbf{0}, r]$, and $T^n(x)\in [\mathbf{0}, r]$ for each $x\in C$ and large enough $n$. Then we need only consider the dynamics of $T$ on $[\mathbf{0}, r]$. However, it is not trivial to check (A3) and to find such an $r$. The following proposition might be helpful. For any $x\in C$ and any nonempty $J\subset I_N$, let $x_J$ be the projection of $x$ onto $C_J$, i.e. $x_J\in C_J$ and $(x_J)_i = x_i$ for all $i\in J$. Then $x_H$ ($x_V$) is the projection of $x$ onto $C_H$ ($C_V$). Denote the $i$th unit vector by $\mathbf{e}_i$, its $i$th component is 1 and others are 0. 

\begin{proposition}\label{pro2.3'}
	Assume (A1) and the existence of continuous $u: \R_+\to C$ such that $\mathbf{0} \ll u(s)\ll u(t)$ for all $t>s\geq 0$ and $u_i(t)\to\infty$ as $t\to \infty$ for all $i\in I_N$. If $\forall t\geq 0, \forall i\in H, \forall j\in V$,
	\begin{equation}\label{E3}
		\max_{s\in[0, u_i(t)]}T_i(s\mathbf{e}_i+u(t)_V) <u_i(t),\quad \max_{s\in[0, u_j(t)]}T_j(s\mathbf{e}_j+u(t)_H) <u_j(t),
	\end{equation}
	then (\ref{e1}) is dissipative and for $r=u(0)$, $T([\mathbf{0}, r])\subset [\mathbf{0}, r)$ and $T^n(x)\in [\mathbf{0}, r]$ for each $x\in C$ and some $n\geq 0$.
\end{proposition}
\begin{proof}
	Note that condition (\ref{E3}) is equivalent to $T([\mathbf{0}, u(t)]) \subset [\mathbf{0}, u(t))$ for all $t\geq 0$. Then the conclusion follows.
\end{proof}

Note that the condition $T([\mathbf{0}, r]) \subset [\mathbf{0}, r]$, which means forward invariance of $[\mathbf{0}, r]$, ensures that the set
\begin{equation}\label{e4}
	\mathcal{A} = \cap_{n=0}^{\infty} T^n([\mathbf{0}, r])
\end{equation}
is nonempty compact invariant ($T(\mathcal{A})= \mathcal{A}$) with $\mathbf{0}\in \mathcal{A}$ and it is the fundamental global attractor of $T$ on $C$, i.e. it attracts the points of any compact set in C uniformly.

By (A2) there is a $q\in (\mathbf{0}, r)$ such that $f(x)\gg\mathbf{1}$ for all $x\in [\mathbf{0}, q]$, $T([\mathbf{0}, r]\setminus [\mathbf{0}, q)) \subset [\mathbf{0}, r]\setminus [\mathbf{0}, q)$ and $[\mathbf{0}, q] \subset T([\mathbf{0}, q])$. As $[\mathbf{0}, r]\setminus [\mathbf{0}, q)$ is a compact subset of $[\mathbf{0}, r]$ and forward invariant, with
\begin{equation}\label{e5}
	\Sigma = \cap_{n=0}^{\infty} T^n([\mathbf{0}, r]\setminus [\mathbf{0}, q)),
\end{equation}
we see that $\Sigma$ is nonempty compact invariant and globally attracts all compact sets in $C\setminus\{\mathbf{0}\}$.

We need to know the detailed composition and geometric features of $\Sigma$, dynamical properties on $\Sigma$, and description of the dynamical behaviour of (\ref{e1}) in terms of $\Sigma$. For this purpose, we establish a link between a type-K competitive system (\ref{e1}) and a type-K retrotone map $T$ defined by (\ref{e2}). 

Since $f$ on $C$ is a $C^1$ map, $T$ is also a $C^1$ map with Jacobian matrix
\begin{equation}\label{e6}
	DT(x) = \textup{diag}(f_1(x), \ldots, f_N(x))(I - M(x)),
\end{equation}
where $I$ is the identity matrix and 
\begin{equation}\label{e7}
	M(x)= (m_{ij}(x))= \left(-\frac{x_i}{f_i(x)}\frac{\partial f_i}{\partial x_j}(x)\right)_{N\times N}.
\end{equation}
Let $\mathcal{M}$ be the set of $N\times N$ matrices of real entries such that $A\in \mathcal{M}$ if and only if $-A$ has the decomposition as $Df$ in (\ref{e3}). Then $M(x)^n\in\mathcal{M}$ for all integers $n\geq 0$ under the assumption (A1). A set $S$ is called {\it type-K convex} if for any distinct $x, y\in S$ satisfying $x<_K y$, the line segment $\overline{xy} = \{tx+(1-t)y: 0\leq t\leq 1\}$ is contained in $S$.

\begin{proposition}\label{pro2.3}
	Assume that the map $T$ defined by (\ref{e2}) satisfies the following conditions: 
	\begin{itemize}
		\item[(a)] (\ref{e1}) is a type-K competitive system. 
		\item[(b)] $T([\mathbf{0}, r])$ is type-K convex.
		\item[(c)] The spectral radius of $M(x)$ satisfies $\rho(M(x)) <1$ for all $x\in [\mathbf{0}, r]$.
	\end{itemize}
	Then $T: [\mathbf{0}, r]\to T([\mathbf{0}, r])$ is a homeomorphism and a weakly type-K retrotone map. If, in addition, 
	\begin{equation}\label{E7}
		\forall i, j\in I_N, \forall x\in [\mathbf{0}, r], \frac{\partial f_i(x)}{\partial x_j} \not= 0,
	\end{equation}
	then $T$ is type-K retrotone.
\end{proposition} 
\begin{proof}
	From condition (c) we see that $I-M(x)$ is invertible for all $x\in [\mathbf{0}, r]$ and 
	\[
	(I-M(x))^{-1} = \sum^{\infty}_{n=0}M(x)^n.
	\]
	From (\ref{e6}) we see that $DT(x)$ is invertible and $(DT(x))^{-1}$ is continuous on $[\mathbf{0}, r]$. So $T$ from $[\mathbf{0}, r]$ to $T([\mathbf{0}, r])$ is a local homeomorphism. As $T(x)=\mathbf{0}$ if and only if $x= \mathbf{0}$, by Lemma 4.1 in \cite{Her}, $T$ from $[\mathbf{0}, r]$ to $T([\mathbf{0}, r])$ is a homeomorphism. 
	
	By conditions (a) and (c) $(I-M(x))^{-1}\in \mathcal{M}$. It then follows from (\ref{e6}) that $(DT(x))^{-1}\in \mathcal{M}$ for all $x\in [\mathbf{0}, r]$. Now for any $x, y\in [\mathbf{0}, r]$ satisfying $T(x)<_K T(y)$, as $T(x), T(y)\in T([\mathbf{0}, r])$, by condition (b) we know that $\overline{T(x)T(y)} \subset T([\mathbf{0}, r])$. Then
	\[
	x-y = T^{-1}(T(x))-T^{-1}(T(y)) = \int^1_0\frac{d T^{-1}(g(s))}{ds}ds,
	\]
	where, by condition (b), $g(s) = T(y) + s(T(x)-T(y)) \in T([\mathbf{0}, r])$ for $s\in [0, 1]$. Since $u = T^{-1}(T(u))$ for $u\in [\mathbf{0}, r]$, we have $Du = I = DT^{-1}(T(u))\times DT(u)$. So 
	\[
	DT^{-1}(T(u)) =(DT(u))^{-1} \in \mathcal{M}.
	\]
	As $T(x)-T(y) <_K \mathbf{0}$, for any matrix $A\in \mathcal{M}$ we have $A(T(x)-T(y)) <_K\mathbf{0}$ and, for any $i\in H$ and $j\in V$, if $T_i(x)-T_i(y)<0$ then $[A(T(x)-T(y))]_i<0$ and if $T_j(x)-T_j(y)>0$ then $[A(T(x)-T(y))]_j>0$. Then
	\[
	\frac{d T^{-1}(g(s))}{ds} = DT^{-1}(g(s))g'(s) = DT^{-1}(g(s))(T(x)-T(y)) <_K\mathbf{0}
	\]
	and $x-y <_K \mathbf{0}$. Moreover, for any $i\in H$ and $j\in V$, if $T_i(x)-T_i(y)<0$ then $x_i - y_i<0$ and if $T_j(x)-T_j(y)>0$ then $x_j-y_j>0$. Thus, $T$ on $[\mathbf{0}, r]$ is weakly type-K retrotone.
	
	If (\ref{E7}) also holds and $T(x)<_K T(y)$, let $J\subset I_N$ such that 
	\[
	j\in J \Longleftrightarrow \max\{T_j(x), T_j(y)\}>0 \Longleftrightarrow \max\{x_j, y_j\}>0 
	\]
	so $x, y, T(x), T(y) \in C_J$. Then (\ref{E7}) implies that, restricted to $C_J$, $\frac{d T^{-1}(g(s))}{ds}\ll_K \mathbf{0}$ and $x-y \ll_K \mathbf{0}$. Hence, $T$ on $[\mathbf{0}, r]$ is type-K retrotone.	
\end{proof}

When $T$ from $[\mathbf{0}, r]$ to $T([\mathbf{0}, r])\subset [\mathbf{0}, r]$ is a homeomorphism, the basin of repulsion of $\mathbf{0}$ is the set
\begin{equation}\label{e8}
	\mathcal{R}(\mathbf{0}) = \{x\in C: \lim_{n\to\infty}T^{-n}(x) = \mathbf{0}\},
\end{equation}
which is an open set of $C$ and invariant. For each $x\in C$ and every $y\in\mathcal{A}$, with overline indicating set closure, the limit sets $\omega(x)$ and $\alpha(y)$ have the usual meaning:
\[
\omega(x) = \bigcap_{n=0}^{\infty}\overline{\{T^m(x): m\geq n\}}, \quad \alpha(y) = \bigcap_{n=0}^{\infty}\overline{\{T^{-m}(y): m\geq n\}}.
\]
For any point $x\in\R^n$ and a number $\delta>0$ we denote by $\mathcal{B}(x, \delta)$ the open ball in $\R^n$ centred at $x$ with radius $\delta$.

\begin{proposition}\label{pro2.4}
	In addition to (A1)--(A3) we assume that $T([\mathbf{0}, r])\subset [\mathbf{0}, r)$ and $T$ from $[\mathbf{0}, r]$ to $T([\mathbf{0}, r])$ is a homeomorphism and weakly type-K retrotone. Then 
	\begin{itemize}
		\item[(a)] The invariant compact set $\Sigma$ defined by (\ref{e5}) satisfies
		\begin{equation}\label{e9}
			\Sigma = \mathcal{A}\setminus \mathcal{R}(\mathbf{0}) = \Sigma_H\cup\Sigma_0\cup\Sigma_V,	
		\end{equation}
		where $\Sigma_0 = \Sigma\setminus(\Sigma_H\cup\Sigma_V)$ with $\Sigma_H\not=\emptyset$, $\Sigma_V\not=\emptyset$ and
		\begin{equation}\label{e10}
			\Sigma_H = \{x\in\Sigma: \alpha(x) \subset (\Sigma\cap C_H)\},\quad \Sigma_V = \{x\in\Sigma: \alpha(x) \subset (\Sigma\cap C_V)\}.	
		\end{equation} 
		\item[(b)] The set $\mathcal{R}(\mathbf{0})$ is connected and type-K convex.
		\item[(c)] For any $x, y\in \mathcal{A}$ with $x <_K y$ and $I =\{m\in I_N: x_m\not= y_m\}$, if $i\in H\cap I$ then $\lim_{n\to\infty}[T^{-n}(x)]_i =0$ and if $j\in V\cap I$ then $\lim_{n\to\infty}[T^{-n}(y)]_j =0$.
	\end{itemize}
	In addition, if $T$ is type-K retrotone, conclusion (c) above is replaced by (d):
	\begin{itemize}
		\item[(d)] For any $x, y\in \mathcal{A}$ with $x <_K y$, we have $\lim_{n\to\infty}[T^{-n}(x)]_H =\mathbf{0}$ and $\lim_{n\to\infty}[T^{-n}(y)]_V = \mathbf{0}$.
	\end{itemize}
\end{proposition}
\begin{proof}
	(a) We first prove $\Sigma = \mathcal{A}\setminus \mathcal{R}(\mathbf{0})$. As $\Sigma$ is invariant and, from (\ref{e5}), $\Sigma\subset \mathcal{A}$ and $\mathbf{0}\not\in\Sigma$, for each $x\in\Sigma$ we have $\alpha(x)\subset\Sigma$ so $\alpha(x)\not= \{\mathbf{0}\}$ and $x\not\in\mathcal{R}(\mathbf{0})$. This shows that $\Sigma \subset \mathcal{A}\setminus \mathcal{R}(\mathbf{0})$. On the other hand, $T^n([\mathbf{0}, q)) \subset \mathcal{R}(\mathbf{0})$ from the choice of $q$, so $(\mathcal{A}\setminus \mathcal{R}(\mathbf{0}))\subset T^n([\mathbf{0}, r]\setminus[\mathbf{0}, q))$ for all integer $n\geq 0$. It follows from this and (\ref{e5}) that $(\mathcal{A}\setminus \mathcal{R}(\mathbf{0}))\subset \Sigma$. Therefore, the equality $\Sigma = \mathcal{A}\setminus \mathcal{R}(\mathbf{0})$ holds.
	
	For each $x\in \mathcal{R}(\mathbf{0})\cap\dot{C}$, as $\mathcal{A}\subset [\mathbf{0}, r)$, there are points $y, z \in [\mathbf{0}, r]\setminus[\mathbf{0}, r)$ such that $y<_K x <_K z$. Then there are $y'\in(\Sigma\cap\overline{xy})$ and $z'\in(\Sigma\cap\overline{xz})$ so that 
	\[
	y<_K y' <_K x <_K z' <_K z. 
	\]
	Since $T$ on $[\mathbf{0}, r]$ is weakly type-K retrotone, we have 
	\[
	T^{-n}(y') <_K T^{-n}(x) <_K T^{-n}(z'), n= 0, 1, 2, \ldots.
	\]
	Thus, for all integers $n\geq 0$,
	\[
	\forall i\in H, 0\leq [T^{-n}(y')]_i \leq [T^{-n}(x)]_i;\quad \forall j\in V,  0\leq [T^{-n}(z')]_j \leq [T^{-n}(x)]_j.
	\]
	These inequalities and $\lim_{n\to\infty}T^{-n}(x) = \mathbf{0}$ imply that $\alpha(y') \subset (\Sigma\cap C_V)$ and $\alpha(z') \subset (\Sigma\cap C_H)$, i.e. $y'\in \Sigma_V$ and $z'\in \Sigma_H$. So $\Sigma_H\not=\emptyset$ and $\Sigma_V\not=\emptyset$.
	
	(b) By (A2) there is a $\delta>0$ such that $(\mathcal{B}(\mathbf{0}, \delta)\cap C) \subset \mathcal{R}(\mathbf{0})$. For any two distinct points $x, y\in\mathcal{R}(\mathbf{0})$, as $\lim_{n\to\infty}T^{-n}(x) =0$ and $\lim_{n\to\infty}T^{-n}(y) =0$, there is an integer $m>0$ such that the line segment $\overline{T^{-m}(x)T^{-m}(y)}$ is contained in $\mathcal{B}(\mathbf{0}, \delta)\cap C$. Then $T^m(\overline{T^{-m}(x)T^{-m}(y)})$ is a curve in $\mathcal{R}(\mathbf{0})$ from $x$ to $y$. Thus, $\mathcal{R}(\mathbf{0})$ is connected. For any two distinct points $u, v\in\mathcal{R}(\mathbf{0})$ satisfying $u<_K v$ we check that $\overline{uv}\subset \mathcal{R}(\mathbf{0})$. Suppose there is a $w\in \overline{uv}\setminus \mathcal{R}(\mathbf{0})$. Since $\mathcal{R}(\mathbf{0})$ is open and $\mathcal{A}$ is closed, we may take $w\in\Sigma$ so $u<_K w <_K v$. As $T$ is weekly type-K retrotone, we have $T^{-n}(u) <_K T^{-n}(w) <_K T^{-n}(v)$ for all $n\geq 0$. Then it follows from $\lim_{n\to\infty}T^{-n}(u) =0$ and $\lim_{n\to\infty}T^{-n}(v) =0$ that $\lim_{n\to\infty}T^{-n}(w) =0$, a contradiction to $w\not\in\mathcal{R}(\mathbf{0})$. Thus, $\overline{uv}\subset \mathcal{R}(\mathbf{0})$ and $\mathcal{R}(\mathbf{0})$ is type-K convex.
	
	(c) Since $T$ is weakly type-K retrotone and $\mathcal{A}$ is invariant, from $x<_K y$ we have $T^{-n}(x) <_K T^{-n}(y)$ for all $n\geq 0$ and
	\[
	\forall i\in H\cap I, [T^{-n}(x)]_i < [T^{-n}(y)]_i; \quad \forall j\in  V\cap I, [T^{-n}(x)]_j > [T^{-n}(y)]_j
	\]
	for all $n\geq 0$. From (A1) we see that $f_i$ is decreasing in the $i$th component, nonincreasing in the first $k$ components and nondecreasing in the last $N-k$ components. Also, $f_j$ is decreasing in the $j$th component, nonincreasing in the last $N-k$ components and nondecreasing in the first $k$ components. Then the above inequalities imply that, for all $n\geq 0$,
	\[
	\forall i\in H\cap I, f_i(T^{-n}(x)) > f_i(T^{-n}(y)); \quad \forall j\in  V\cap I, f_j(T^{-n}(x)) < f_j(T^{-n}(y)).
	\]
	Suppose $\lim_{n\to\infty}[T^{-n}(x)]_i \not= 0$ for some $i\in H\cap I$. Then $x_i>0$ and $0< [T^{-n}(x)]_i < [T^{-n}(y)]_i$ for all $n\geq 0$. By the compactness of $\mathcal{A}$ and $T^{-n}(x), T^{-n}(y) \in\mathcal{A}$, there exist an increasing sequence $\{\sigma(n)\}\subset\N$ and $\tilde{x},\tilde{y} \in\mathcal{A}$ with $\tilde{x} \leq_K \tilde{y}$ and $\tilde{x}_i>0$ such that $\lim_{n\to\infty}T^{-\sigma(n)}(x) = \tilde{x}$ and $\lim_{n\to\infty}T^{-\sigma(n)}(y) = \tilde{y}$. Then, since $\Delta(n) = \frac{[T^{-n}(x)]_i}{[T^{-n}(y)]_i} <1$ for all $n\in\N$,
	\[
	\Delta(n) = \frac{T_i(T^{-(n+1)}(x))}{T_i(T^{-(n+1)}(y))} = \Delta(n+1) \frac{f_i(T^{-(n+1)}(x))}{f_i(T^{-(n+1)}(y))} >\Delta(n+1). 
	\]
	As $\Delta(n)$ is decreasing and $\lim_{n\to\infty}\Delta(\sigma(n))= \frac{\tilde{x}_i}{\tilde{y}_i}$, we have $\lim_{n\to\infty}\Delta(n) = \frac{\tilde{x}_i}{\tilde{y}_i} \in (0, 1)$ and 
	\[
	\frac{\tilde{x}_i}{\tilde{y_i}} =\lim_{n\to\infty}\Delta(\sigma(n)-1) =\lim_{n\to\infty} \frac{T_i(T^{-\sigma(n)}(x))}{T_i(T^{-\sigma(n)}(y))} = \frac{T_i(\tilde{x})}{T_i(\tilde{y})} =\frac{\tilde{x}_i}{\tilde{y_i}}\frac{f_i(\tilde{x})}{f_i(\tilde{y})}.
	\] 
	It immediately follows that $f_i(\tilde{x}) = f_i(\tilde{y})$. But from (A1), $\tilde{x}<_K \tilde{y}$ and $\tilde{x}_i < \tilde{y}_i$ we should have $f_i(\tilde{x}) > f_i(\tilde{y})$, a contradiction. Therefore, we must have $\lim_{n\to\infty}[T^{-n}(x)]_i = 0$ for $i\in H\cap I$. By the same reasoning as above, we also have $\lim_{n\to\infty}[T^{-n}(y)]_j =0$ for all $j\in V\cap I$.
	
	(d) If $T$ is type-K retrotone, for any $x, y\in\mathcal{A}$ with $x<_K y$ and $J = \{j\in I_N: \max\{x_j, y_j\}>0\}$, we have $x, y, T^{-n}(x), T^{-n}(y) \in C_J$. Then restricting to $C_J$ we obtain $T^{-n}(x)\ll_K T^{-n}(y)$ for all $n>0$. The conclusion follows from applying the above reasoning to each $i\in H\cap J$ and $j\in V\cap J$.
\end{proof}

Note that the type-K convexity of a set means that, for any straight line that distinct points on it are related by $<_K$, the intersection of the line with the set (if any) is only ONE line segment. This implies that the set is ``solid'' with no ``air bubbles'' completely surrounded by the points of the set.

\begin{proposition}\label{pro2.5}
	The following hold under the assumptions of Proposition \ref{pro2.4}. 
	\begin{itemize}
		\item[(a)] The set $\Sigma_H\cap C_H = \Sigma\cap C_H$ ($\Sigma_V\cap C_V = \Sigma\cap C_V$) is the global attractor of the system restricted to $C_H\setminus\{\mathbf{0}\}$ ($C_V\setminus\{\mathbf{0}\}$), and it is a repeller on $\Sigma$.The set $\Sigma_0$ is nonempty compact invariant and it is the global attractor of the system on $C\setminus(C_H\cup C_V)$. 
		\item[(b)] The sets $\Sigma_H\cup\Sigma_0$ and $\Sigma_H\cup\Sigma_0$ are unordered in $\ll_K$.
	\end{itemize}
	If $T$ is type-K retrotone, then ``$\ll_K$'' in conclusion (b) is replaced by ``$<_K$''.
\end{proposition}

Then proof of Proposition \ref{pro2.5} is left to Section 4. From these results and \cite{Hou3}, we are tempted to expect that $\Sigma$ for system (\ref{e1}) would have the same features as those of the global attractor of a type-K competitive differential system restricted to $C\setminus\{\mathbf{0}\}$. What are the geometric and dynamical features of $\Sigma$? Is $\Sigma_H$, $\Sigma_V$ or $\Sigma$ an $(N-1)$-dimensional surface? Is each forward half orbit of (\ref{e1}) in $C\setminus\{\mathbf{0}\}$ asymptotic to one in $\Sigma$? There are many intriguing problems waiting for our exploration. Since this paper is mainly dealing with planar systems, further investigation of $N$-dimensional type-K competitive systems will not be included here.

\section{Global dynamics of planar type-K competitive systems}
In this section, we focus our attention to the planar system (\ref{e1}) with $N=2$, where $T$ in (\ref{e2}) is rewritten as 
\begin{equation}\label{e11}
	T_1(x) = x_1f_1(x_1, x_2), T_2(x) = x_2f_2(x_1, x_2), x =(x_1, x_2) \in \R^2_+ = C.
\end{equation}

\begin{proposition}\label{pro3.1}
	Assume that (\ref{e1}) with (\ref{e11}) is type-K competitive and $\rho(M(x))<1$ for all $x\in [\mathbf{0}, r]$. Then $T([\mathbf{0}, r])$ is type-K convex.
\end{proposition}
The proof of this proposition is postponed to Section 5. 
\begin{remark}
	From Proposition \ref{pro3.1} we see that for a planar system the condition (b) in Proposition \ref{pro2.3} is redundant. Thus, the conditions of  Proposition \ref{pro3.1} ensure that $T$ is a homeomorphism from $[\mathbf{0}, r]$ to $T([\mathbf{0}, r])$ and a weakly type-K retrotone map. Moreover, if (\ref{E7}) holds then $T$ is type-K retrotone.
\end{remark} 

For $N=2$ we can easily derive a necessary and sufficient condition for $\rho(M(x))<1$. Indeed, $\det(M(x) -\rho I) = 0$ if and only if
\[
\rho^2 - \textup{tr}(M(x))\rho + \det(M(x)) =0,
\]
which gives
\[
2\rho = \textup{tr}(M(x)) \pm \sqrt{[\textup{tr}(M(x))]^2 - 4\det(M(x))}.
\]
By (A1) we have $\textup{tr}(M(x)) >0$, and from (\ref{e6}) and (\ref{e7}) we can check that $[\textup{tr}(M(x))]^2 - 4\det(M(x))>0$, for $x\in C\setminus\{\mathbf{0}\}$. Then $\rho(M(x))<1$ if and only if 
\begin{equation}\label{e12}
	\textup{tr}(M(x)) <\min\{2, 1+\det(M(x))\}.
\end{equation}
In some circumstances, (\ref{e12}) might be convenient to use for $\rho(M(x)) <1$.

A curve $\ell\subset \R^N$ is called monotone in $<$ $(\ll, <_K, \ll_K)$ if every pair of distinct points on $\ell$ are related by $<$ $(\ll, <_K, \ll_K)$.

\begin{theorem}\label{the3.2}
	For (\ref{e1}) with (\ref{e11}) assume (A2) and the existence of $r\gg\mathbf{0}$ such that $\omega(x)\subset [\mathbf{0}, r]$ for all $x\in C$, $T([\mathbf{0}, r])\subset [\mathbf{0}, r)$ and $\rho(M(x))<1$ for all $x\in [\mathbf{0}, r]$. Assume also that (\ref{e1}) with (\ref{e11}) is type-K competitive on $[\mathbf{0}, r]$. Then we have the following conclusions.
	\begin{itemize}
		\item[(a)] The system has one fixed point $Q_i$ on each positive half $x_i$-axis, which is a saddle point with its unstable manifold $W^u(Q_i)\subset\dot{C}$ and the positive half $x_i$-axis as its stable manifold $W^s(Q_i)$.
		\item[(b)] The map $T$ has at least one fixed point $p_0$ in $\dot{C}$. If $p_0$ is the unique fixed point in $\dot{C}$ then it is globally asymptotically stable in $\dot{C}$. Otherwise, $T$ has at least another fixed point $p_1\gg p_0$ such that $T^n(x)$ is eventually monotone in $\ll_K$ or $<$ and converges to a fixed point in $[p_0, p_1]$ for every $x\in\dot{C}$.
		\item[(c)] For the global attractor $\Sigma = \mathcal{A}\setminus\mathcal{R}(\mathbf{0})$ on $C\setminus\{\mathbf{0}\}$ with the decomposition $\Sigma = \Sigma_H\cup\Sigma_0\cup\Sigma_V$ (see Figure 1), we have
		\[
		\Sigma_H = W^u(Q_1)\cup\{Q_1\}, \quad \Sigma_V = W^u(Q_2)\cup\{Q_2\},
		\]
		both are monotone curves in $<$, $\Sigma_0 =\{p_0\}$ or $\Sigma_0$ is a monotone curve in $\ll$ from $p_0$ to $p_1$. If $\frac{\partial f_1(x)}{\partial x_2} >0$ and $\frac{\partial f_2(x)}{\partial x_1} >0$ for $x\in [\mathbf{0}, r]$ then $\Sigma_H$ and $\Sigma_V$ are monotone in $\ll$.
	\end{itemize}
\end{theorem}
\begin{figure}[htb]
	\includegraphics[width=5.5in,height=3in]{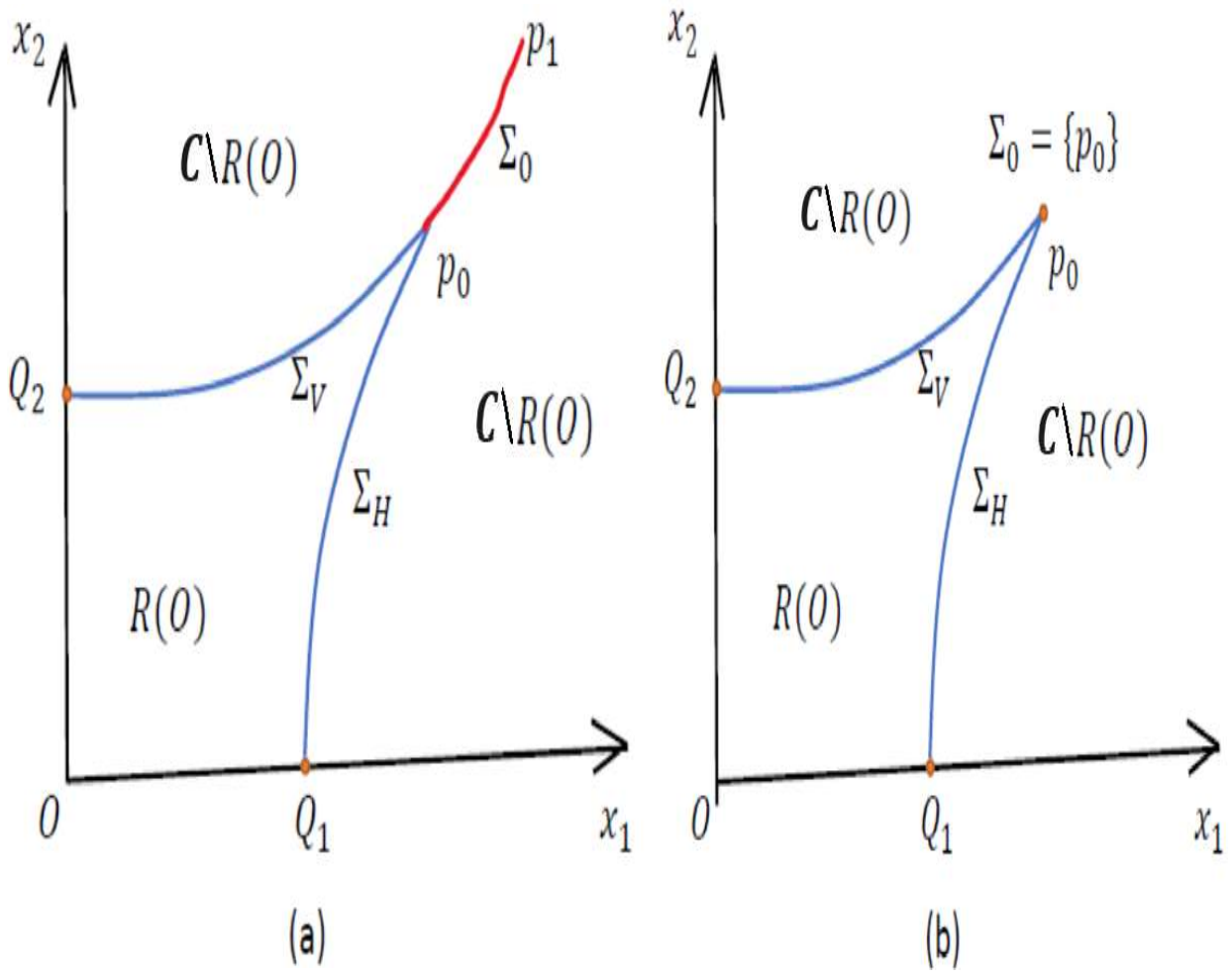}
	\caption{Illustration of $\Sigma = \Sigma_H\cup\Sigma_0\cup\Sigma_V$ when $N=2$: (a) $\Sigma_0$ is a monotone curve in $\ll$, (b) $\Sigma_0 =\{p_0\}$. }
\end{figure}

The proof of Theorem \ref{the3.2} will be given in Section 5. In the rest of this section, we illustrate Theorem \ref{the3.2} by the following example.

\vskip 3 mm \textbf{Example 1} 
	Investigate the dynamics of system (\ref{e1}) with (\ref{e11}) where
	\begin{equation}\label{e21}
		T(x_1, x_2) = \left(\begin{array}{l}
			x_1(1+ b\tan^{-1}(x_2-1 -a(x_1-1) -(x_1-1)^3)) \\
			x_2(1+ b\tan^{-1}(x_1-1 -a(x_2-1) -(x_2-1)^3))
		\end{array}\right),
	\end{equation}
	$a$ and $b$ are positive constants satisfying $0<b< \frac{2}{\pi}$ and other conditions specified later.
\vskip 3 mm 
Then $f_1(x_1, x_2) = 1+ b\tan^{-1}(x_2-1 -a(x_1-1) -(x_1-1)^3) \in (0, 2)$ and $f_2(x_1, x_2) = f_1(x_2, x_1) \in (0, 2)$. The curves $L_1 = \{x\in C: f_1(x)=1\}$ and $L_2 = \{x\in C: f_2(x)=1\}$ are the graphs of the functions
\begin{equation} \label{e22}
	x_2 = 1+ a(x_1-1) + (x_1-1)^3, \quad 	x_1 = 1+ a(x_2-1) + (x_2-1)^3
\end{equation}
in $C$ respectively. As the two curves $L_1$ and $L_2$ and the two functions $f_1$ and $f_2$ are symmetric about the half-line $x_2=x_1\geq 0$, this half-line is a forward invariant manifold and the dynamics of the system restricted to this one-dimensional manifold is described by 
\begin{equation}\label{e23}
	u(n) = g^n(u),\;\; g(u) = u(1+b\tan^{-1}((u-1)(1-a-(u-1)^2))),\;\; u\in \R_+.
\end{equation}
Assume that 
\begin{equation}\label{e24}
	0< b< \min\{\frac{a}{\pi}, \frac{1}{8+2a+ \tan^{-1}(2+a)}\}.
\end{equation}
If $a\geq 1$ then the map $g: \R_+\to\R_+$ has two fixed points: $u = 0$ is a repeller and $u = 1$ is an attractor. If $0 < a <1$ then $g$ has four fixed points: $u=0$ and $u=1$ are repellers and $u = 1-\sqrt{1-a}$ and $u = 1+\sqrt{1-a}$ are attractors. From these we obtain all the fixed points of (\ref{e1}) with (\ref{e21}): $(0, 0), (1, 1), (s_0, 0), (0, s_0)$ and, if $0 <a< 1$, $(1-\sqrt{1-a}, 1-\sqrt{1-a})$ and $(1+\sqrt{1-a}, 1+\sqrt{1-a})$. Here $s_0\in (0, 1)$ is the unique solution of the equation $(s-1)^3 +a(s-1) +1=0$.

Next, we check that $T$ given by (\ref{e21}) satisfies all the conditions of Theorem \ref{the3.2}. As $f_1(0, 0) = f_2(0, 0) = 1+b\tan^{-1}(a)>1$, the condition (A2) holds. We can easily check that $\frac{\partial f_i}{\partial x_i}<0$ and $\frac{\partial f_i}{\partial x_j}>0$ for all $i, j\in\{1, 2\}$ ($i\not= j$) so (A1) holds. We now use Proposition \ref{pro2.3'} to show that $r =(2, 2)\gg\mathbf{0}$ meets all the requirement of Theorem \ref{the3.2}. Indeed, for $u(t) = (t+2, t+2)$ with each fixed $t\geq 0$, we first check that
\begin{equation}\label{e25}
	\max\{T_1(x_1, t+2): 0\leq x_1\leq t+2\} <t+2.
\end{equation}   
For $0\leq x_1\leq 1$, by (\ref{e24}), 
\[
T_1(x_1, t+2) \leq 1 +b\tan^{-1}(t+2+a) < 1 +\frac{\pi}{2}b < 2+t.
\]
For $x_1\geq 1$, $f_1(x_1, t+2)$ is decreasing with $f_1(1, t+2) = 1+b\tan^{-1}(t+1)>1$ and $f_1(t+2, t+2) = 1 +b\tan^{-1}((t+1)(1-a-(t+1)^2)) <1$ so there is a unique $x^*_1 \in (1, t+2)$ such that $f_1(x^*_1, t+2)=1$. Then, for $x_1\in (x^*_1, t+2]$, $T_1(x_1, t+2) <x_1\leq t+2$. So we need only check that $T_1(x_1, t+2) <t+2$ for $1\leq x_1\leq x^*_1$. As
\begin{eqnarray*}
	\frac{\partial T_1(x_1, t+2)}{\partial x_1} &=& 1 +b\tan^{-1}(t+1-a(x_1-1)-(x_1-1)^3) \\
	&& - \frac{bx_1(a+3(x_1-1)^2)}{1+[t+1-a(x_1-1)-(x_1-1)^3]^2}
\end{eqnarray*}
is decreasing, if $x^*_1\leq 2$ then
\[
\frac{\partial T_1}{\partial x_1}(x^*_1, t+2) = 1-bx^*_1(a+3(x^*_1-1)^2) \geq 1- 2b(3+a) >0
\]
by (\ref{e24}). Thus, $T_1(x_1, t+2)$ for $x_1\in [1, x^*_1]$ is increasing and $T_1(x_1, t+2) \leq T_1(x^*_1, t+2) = x^*_1 < t+2$. If $x^*_1 >2$ then $t>a$. By (\ref{e24}) again,
\[
\frac{\partial T_1}{\partial x_1}(2, t+2) = 1+b\tan^{-1}(t-a) -\frac{2b(3+a)}{1+(t-a)^2} > b\tan^{-1}(t-a)>0.
\]
Hence, for $1\leq x_1\leq 2$, 
\[
T_1(x_1, t+2) \leq T_1(2, t+2) = 2(1+b\tan^{-1}(t-a))<2 + t-a< t+2. 
\]
Let $\bar{x}_1 = \frac{t+2}{1+b\tan^{-1}(t-a)}$. As $\bar{x}_1 >2$, for $2\leq x_1< \bar{x}_1$ we have 
\[
T_1(x_1, t+2)< \bar{x}_1f_1(2, t+2) = \bar{x}_1(1+b\tan^{-1}(t-a)) = t+2.
\]
From (\ref{e24}) and $0< \tan^{-1}(t-a) < \frac{\pi}{2}$ we have $a> b\pi > 2b\tan^{-1}(t-a)$, so
\[
a(a+1) > 2(a+1)b\tan^{-1}(t-a),
\]
which is equivalent to
\[
a(a-b\tan^{-1}(t-a)) > (a+2)(1+b\tan^{-1}(t-a)) -2(a+1).
\]
As $a-b\tan^{-1}(t-a)>0$ and $t>a$, we must have
\[
t(a-b\tan^{-1}(t-a)) > (a+2)(1+b\tan^{-1}(t-a)) -2(a+1),
\]
which is the same as
\[
(a+1)t+2(a+1) > (a+2+ t)(1+b\tan^{-1}(t-a)).
\]
Dividing through by $(1+b\tan^{-1}(t-a))$ we have $(a+1)\bar{x}_1 > (a+1) +t+1$, so 
\[
t+1 < (a+1)(\bar{x}_1-1) < a(\bar{x}_1-1) + (\bar{x}_1-1)^3.
\]
This shows that $f_1(\bar{x}_1, t+2) <1 = f_1(x^*_1, t+2)$ so $x^*_1 < \bar{x}_1$. Therefore, for $2\leq x_1\leq x^*_1$, $T_1(x_1, t+2) < t+2$ and (\ref{e25}) follows. As $T_2(t+2, x_2) = T_1(x_2, t+2)$, from (\ref{e25}) it follows that $\max\{T_2(t+2, x_2): 0\leq x_2\leq t+2\}< t+2$. By proposition \ref{pro2.3'}, for $r = (2, 2)$ we have $T([\mathbf{0}, r]) \subset [\mathbf{0}, r)$ and $\omega(x)\subset [\mathbf{0}, r]$ for all $x\in C$.

Finally, we check that $\rho(M(x)) <1$ for $x\in [\mathbf{0}, r]$. Note that $M(x)$ has entries $m_{11}(x) = -\frac{x_1}{f_1(x)}\frac{\partial f_1(x)}{\partial x_1}$, $m_{12}(x) = -\frac{x_1}{f_1(x)}\frac{\partial f_1(x)}{\partial x_2}$, $m_{21}(x) = -\frac{x_2}{f_2(x)}\frac{\partial f_2(x)}{\partial x_1}$ and $m_{22}(x) = -\frac{x_2}{f_2(x)}\frac{\partial f_2(x)}{\partial x_2}$. As 
\[
|m_{i1}(x)| + |m_{i2}(x)| \leq \frac{2(a+4)b}{1-b\tan^{-1}(2+a)}
\]
for $i = 1, 2$, from (\ref{e24}) we have $\|M(x)\| \leq \frac{2(a+4)b}{1-b\tan^{-1}(2+a)} <1$. By Theorem 6.1.3 in \cite{Lan}, $\rho(M(x)) \leq \|M(x)\| <1$. Therefore, (\ref{e1}) with (\ref{e21}) and (\ref{e24}) and $r = (2, 2)$ satisfies all the conditions of Theorem \ref{the3.2}. Then we obtain the following conclusions:
\begin{itemize}
	\item[(i)] If $a\geq 1$ then the fixed point $p_0 = (1, 1)$ is globally asymptotically stable in $\dot{C}$, the fixed points $Q_1 = (s_0, 0)$ and $Q_2 = (0, s_0)$ are saddle points, and the global attractor $\Sigma = \Sigma_H\cup\Sigma_0\cup\Sigma_V$ on $C\setminus\{\mathbf{0}\}$ is shown in Figure 1 (b). 
	\item[(ii)] If $0 < a< 1$ the fixed points $p_0 = (1-\sqrt{1-a}, 1-\sqrt{1-a})$ and $p_1 = (1+\sqrt{1-a}, 1+\sqrt{1-a})$ are attractors but $p_2 = (1, 1)$ is a saddle point; every forward orbit in $\dot{C}$ converges to some $p_i$; the global attractor $\Sigma$ is shown in Figure 1 (a).
\end{itemize}

\section{Proof of Proposition \ref{pro2.5}}
\begin{proof}
	(b) The conclusion follows from Proposition \ref{pro2.4} (c) and (d) immediately.
	
	(a) Since $C_H\setminus \{\mathbf{0}\}$ and $C_V\setminus\{\mathbf{0}\}$ are forward invariant and $\Sigma$ is the global attractor of the system on $C\setminus \{\mathbf{0}\}$, it follows that $\Sigma\cap C_H = \Sigma_H\cap C_H$ ($\Sigma\cap C_V = \Sigma_V\cap C_V$) is the global attractor of the system on $C_H\setminus \{\mathbf{0}\}$ ($C_V\setminus \{\mathbf{0}\}$). 
	
	From (A2) we can choose $\varepsilon>0$ and $\delta>0$ small enough such that 
	\[ \forall x\in S(\mathbf{0}, \varepsilon) =\{y\in C: \|y_H\|<\varepsilon, \|y_V\|<\varepsilon\}, f(x)\gg (1+ \delta)\mathbf{1}
	\] 
	so that $S(\mathbf{0}, \varepsilon) \subset T(S(\mathbf{0}, \varepsilon)) \subset [\mathbf{0}, r)$. As $T$ on $[\mathbf{0}, r]$ is a homeomorphism, we have $T([\mathbf{0}, r]\setminus S(\mathbf{0}, \varepsilon)) \subset [\mathbf{0}, r]\setminus S(\mathbf{0}, \varepsilon)$. By the forward invariance of $C_H\setminus \{\mathbf{0}\}$ and $C_V\setminus\{\mathbf{0}\}$, restricting the set inclusion to $C_H$ and $C_V$ gives
	\begin{eqnarray}
		T(C_H\cap([\mathbf{0}, r]\setminus S(\mathbf{0}, \varepsilon))) &\subset& C_H\cap([\mathbf{0}, r]\setminus S(\mathbf{0}, \varepsilon)), \label{e15} \\ T(C_V\cap([\mathbf{0}, r]\setminus S(\mathbf{0}, \varepsilon))) &\subset& C_V\cap([\mathbf{0}, r]\setminus S(\mathbf{0}, \varepsilon)). \label {e16}
	\end{eqnarray} 
	Now consider the sets $S(C_H, \varepsilon) = \{x\in [\mathbf{0}, r]: \|x_V\|<\varepsilon\}$ and $S(C_V, \varepsilon) = \{x\in [\mathbf{0}, r]: \|x_H\|<\varepsilon\}$. Clearly, $S(C_H, \varepsilon)\cap S(C_V, \varepsilon) = S(\mathbf{0}, \varepsilon)$. We now verify that
	\begin{eqnarray}
		T([\mathbf{0}, r]\setminus S(C_H, \varepsilon)) &\subset& [\mathbf{0}, r]\setminus S(C_H, \varepsilon), \label{e17} \\
		T([\mathbf{0}, r]\setminus S(C_V, \varepsilon)) &\subset& [\mathbf{0}, r]\setminus S(C_V, \varepsilon). \label{e18} 
	\end{eqnarray}
	For each $x\in [\mathbf{0}, r]\setminus S(C_H, \varepsilon)$, we have $\|x_V\|\geq \varepsilon$ and $x_V\in C_V\cap ([\mathbf{0}, r]\setminus S(C_H, \varepsilon)) \subset C_V\cap ([\mathbf{0}, r]\setminus S(\mathbf{0}, \varepsilon))$. By (A1) we have $\|T(x)_V\|\geq \|T(x_V)_V\|$. As $T(x_V)_V = T(x_V)$, from (\ref{e16}) we obtain $\|T(x_V)_V\| \geq \varepsilon$. Thus, $T(x) \in [\mathbf{0}, r]\setminus S(C_H, \varepsilon)$ and (\ref{e17}) follows. Similarly, (\ref{e18}) follows from (A1) and (\ref{e15}). 
	
	For each $x\in S(C_H, \varepsilon)$ and any $j\in V$ with $x_j>0$, by $\|x_V\|<\varepsilon$, $x_V\in S(\mathbf{0}, \varepsilon)$ and (A1) we have $T_j(x) = x_jf_j(x) \geq x_jf_j(x_V) >(1+\delta)x_j$. In parallel, for each $x\in S(C_V, \varepsilon)$ and any $i\in H$ with $x_i>0$, by $\|x_H\|<\varepsilon$, $x_H\in S(\mathbf{0}, \varepsilon)$ and (A1) we have $T_i(x) = x_if_i(x) \geq x_if_i(x_H) >(1+\delta)x_i$. This together with (\ref{e17}) and (\ref{e18}) shows that, for each $x\in (S(C_H, \varepsilon) \cup S(C_V, \varepsilon))\setminus(C_H\cup C_V)$, $T^n(x)\in [\mathbf{0}, r]\setminus (S(C_H, \varepsilon)\cup S(C_H, \varepsilon))$ for some $n>0$. In particular, $\Sigma\cap C_H$ repels the points of $\Sigma_H\setminus C_H$ away and $\Sigma\cap C_V$ repels the points of $\Sigma_V\setminus C_V$ away. So $\Sigma\cap C_H$ and $\Sigma_H\setminus C_H$ are repellers on $\Sigma$.
	
	Combination of (\ref{e17}) and (\ref{e18}) gives
	\[
	T([\mathbf{0}, r]\setminus (S(C_H, \varepsilon)\cup S(C_H, \varepsilon))) \subset [\mathbf{0}, r]\setminus (S(C_H, \varepsilon)\cup S(C_V, \varepsilon)). 
	\]
	It then follows from this and the compactness of the sets that the set
	\begin{equation}\label{e19}
		S_0 = \bigcap_{n=0}^{\infty}T^n([\mathbf{0}, r]\setminus (S(C_H, \varepsilon)\cup S(C_V, \varepsilon)))
	\end{equation}
	is nonempty, compact and invariant subset of $\Sigma$. As $\varepsilon$ can be taken arbitrarily small, (\ref{e19}) shows that $S_0$ attracts the points of any compact set in $C\setminus (C_H\cup C_V)$ uniformly. 
	
	We now show that $S_0 = \Sigma_0$. Clearly, $S_0\cap (C_H\cup C_V) = \emptyset$. If there is a point $x\in S_0\cap (\Sigma_H\cup\Sigma_V)$ then $\alpha(x)\subset (C_H\cup C_V)$, a contradiction to $\alpha(x)\subset S_0$ by the invariance of $S_0$. Thus $S_0\cap (\Sigma_H\cup\Sigma_V) =\emptyset$. By $S_0\subset \Sigma$ we must have $S_0\subset \Sigma_0$. Conversely, each bounded whole orbit is invariant and $\Sigma_0$ consists of the whole orbits that are not in $\Sigma_H\cup\Sigma_V\cup\mathcal{R}(\mathbf{0})$. If there exists $x\in\Sigma_0\cap (S(C_H, \varepsilon)\cup S(C_V, \varepsilon))$, by forward invariance of $[\mathbf{0}, r]\setminus (S(C_H, \varepsilon)\cup S(C_V, \varepsilon))$ we must have $T^{-n}(x) \in \Sigma_0\cap (S(C_H, \varepsilon)\cup S(C_V, \varepsilon))$ for all $n\geq 0$. If $\|x_V\|< \varepsilon$ with $x_j>0$ for some $j\in V$ then from the analysis above we have $(1+\delta)T^{-1}(x)_j<x_j$ and $T^{-n}(x)_j < (1+\delta)^{-n}x_j \to 0$ as $n\to\infty$. Thus, $\|x_V\|< \varepsilon$ implies $\alpha(x) \subset C_H$. Similarly, $\|x_H\|< \varepsilon$ implies $\alpha(x) \subset C_V$. But $\alpha(x) \subset (C_H\cup C_V)$ implies the whole orbit $\gamma(x)$ in $(\Sigma_H\cup\Sigma_V\cup\mathcal{R}(\mathbf{0}))$, a contradiction to $\gamma(x)\subset \Sigma_0$. This shows that $\Sigma_0 \subset ([\mathbf{0}, r]\setminus (S(C_H, \varepsilon)\cup S(C_V, \varepsilon)))$. By the invariance of $\Sigma_0$ and (\ref{e19}), $\Sigma_0\subset S_0$.  
\end{proof}

\section{Proof of Proposition \ref{pro3.1} and Theorem \ref{the3.2}}
We need the following lemma. Let $\ell_1$ ($\ell_2$) be the lower and right (upper and left) boundary of $[\mathbf{0}, r]$ so that $[\mathbf{0}, r]\setminus (\mathbf{0}, r) = \ell_1\cup\ell_2$ and $\ell_1\cap\ell_2 = \{\mathbf{0}, r\}$.
\begin{lemma}\label{lem5.1}
	Assume that (\ref{e1}) with (\ref{e11}) is type-K competitive and $\rho(M(x))<1$ for all $x\in [\mathbf{0}, r]$. Then, for any distinct $T(x), T(y) \in T([\mathbf{0}, r])$ with $T(x)<_K T(y)$, if $\overline{T(x)T(y)}\subset T([\mathbf{0}, r])$ then $T^{-1}(\overline{T(x)T(y)})\subset [\mathbf{0}, r]$ is a monotone curve in $<_K$ with $x <_K y$. Moreover, if the line segment $\overline{T(x)T(y)}$ is non-extendable in $T([\mathbf{0}, r])$ from both ends, then $x\in\ell_2$ and $y\in\ell_1$. If a monotone line in $<_K$ has a touching point $T(x)$ only with $T([\mathbf{0}, r])$ then either $x=r$ or $x=\mathbf{0}$.	
\end{lemma}
\begin{proof}
	Since $T(x), T(y) \in T([\mathbf{0}, r])$ with $T(x)<_K T(y)$, if $\overline{T(x)T(y)}\subset T([\mathbf{0}, r])$ from the proof of Proposition \ref{pro2.3} we know that $T$ is a homeomorphism from $[\mathbf{0}, r]$ to $T([\mathbf{0}, r])$ and that $T^{-1}(\overline{T(x)T(y)})\subset [\mathbf{0}, r]$ is a monotone curve in $<_K$ with $x <_K y$. 
	
	Now suppose $\overline{T(x)T(y)}$ is non-extendable in $T([\mathbf{0}, r])$ from both ends. We show that $x\in \ell_2$ and $y\in\ell_1$. If this is not true, then either $x\not\in\ell_2$ or $y\not\in\ell_1$, and a contradiction will be derived below in either case. Suppose $x\not\in\ell_2$. Then $T(x)\not\in T(\ell_2)$. As $T$ is a homeomorphism and $\overline{T(x)T(y)}$ is non-extendable in $T([\mathbf{0}, r])$ from both ends, $T(x)$ and $T(y)$ must be on the boundary $(T(\ell_1)\cup T(\ell_2))$ of $T([\mathbf{0}, r])$. So $T(x)\in T(\ell_1\setminus\{\mathbf{0}, r\})$ and $x\in (\ell_1\setminus\{\mathbf{0}, r\})$. Then $x<_K y$ and $y\in [\mathbf{0}, r]$ imply $y\in (\ell_1\setminus\{\mathbf{0}, r\})$. So either both $x, y$ are on the $x_1$-axis with $0< x_1 < y_1\leq r_1, x_2=y_2 =0$ or both are on the vertical line with $x_1 =y_1= r_1, r_2>x_2 > y_2 \geq 0$. Thus, $\overline{xy}$ is the only one monotone curve in $<_K$ from $x$ to $y$ and $T^{-1}(\overline{T(x)T(y)}) = \overline{xy} \subset (\ell_1\setminus\{\mathbf{0}, r\})$. Hence, $\overline{T(x)T(y)} \subset T(\ell_1\setminus\{\mathbf{0}, r\})$. As $\overline{T(x)T(y)}$ is non-extendable in $T([\mathbf{0}, r])$, the region $T([\mathbf{0}, r])\setminus \overline{T(x)T(y)}$ is completely on one side of the straight line (see Figure 2). Now take a point $p\in \ell_1$ close to $x$ such that $p<_K x$. Then $T(p)\in (T(\ell_1)\setminus \overline{T(x)T(y)})$. The line through $T(p)$ parallel to $\overline{T(x)T(y)}$ intersects the region. So there is some $T(q)\in T(\ell_1\cup\ell_2)$ such that $\overline{T(p)T(q)}\subset T([\mathbf{0}, r])$ and $\overline{T(p)T(q)}$ is parallel to $\overline{T(x)T(y)}$ with $T(p) <_K T(q)$. Then $p<_K q$. This, together with $p\in\ell_1$ and $q\in [\mathbf{0}, r]$ implies that $q\in \ell_1$ and $T^{-1}(\overline{T(p)T(q)}) = \overline{pq}$. By the choice of $p$ we see that $p, x, y, q$ are on the same straight line and $q\to y$ if $p\to x$. Thus, $\overline{xy}\cap\overline{pq}\not=\emptyset$, which is equivalent to $\overline{T(x)T(y)}\cap\overline{T(p)T(q)}\not=\emptyset$, when $p$ is close to $x$. This contradicts $\overline{T(p)T(q)} \subset (T([\mathbf{0}, r])\setminus \overline{T(x)T(y)})$ on one side of $\overline{T(x)T(y)}$. Therefore, we must have $x\in\ell_2$. If $y\not\in\ell_1$, by the same reasoning as above we would have $\overline{xy} = T^{-1}(\overline{T(x)T(y)}) \subset \ell_2\setminus\{\mathbf{0}, r\}$ and derive a contradiction. 
	
	\begin{figure}[htb]
		\includegraphics[width=5.5in,height=4in]{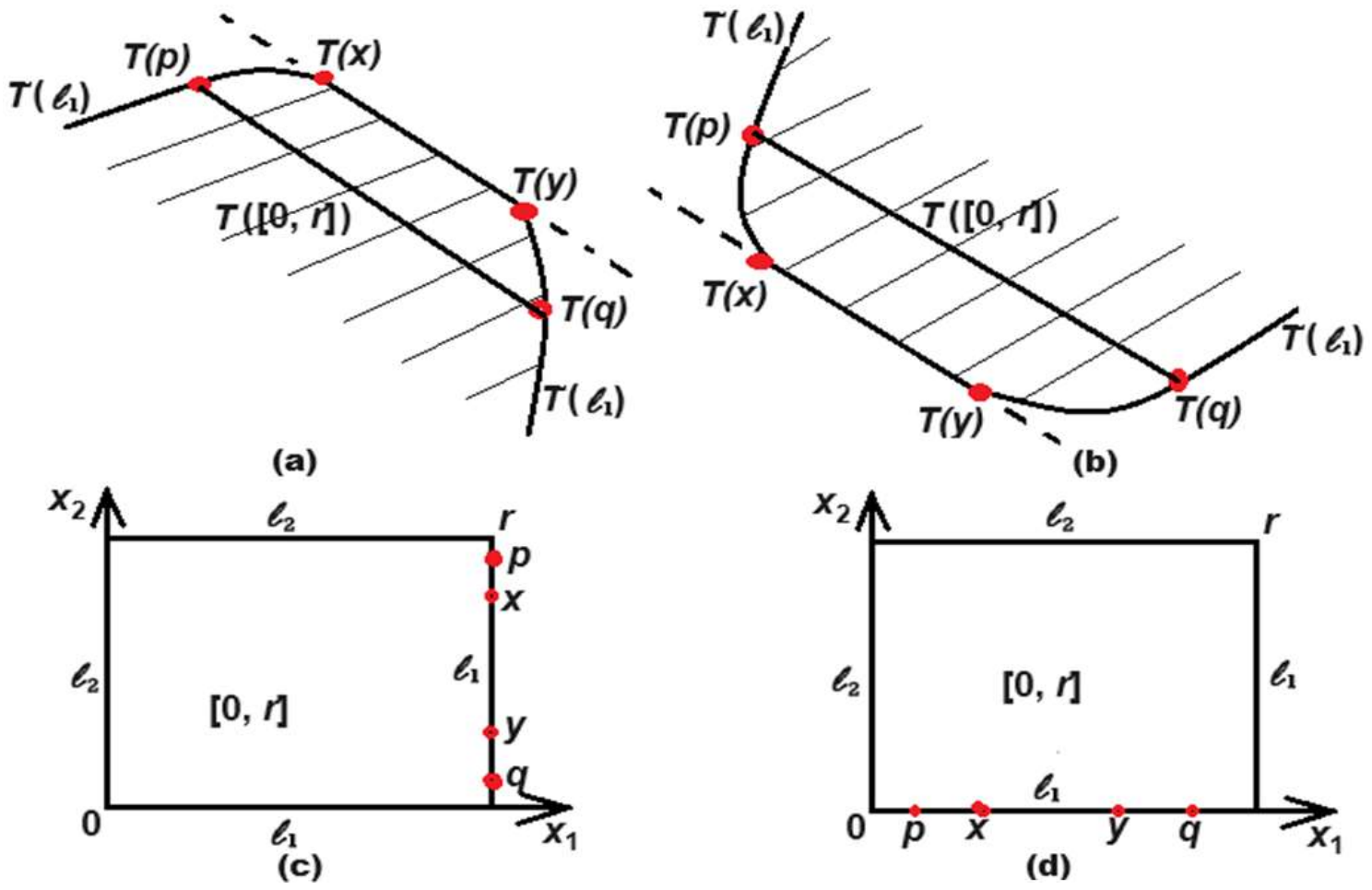}
		\caption{Illustration of $\overline{xy}$, $\overline{pq}$, $\overline{T(x)T(y)}$ and $\overline{T(p)T(q)}$: (a) \& (b) The region $T([\mathbf{0}, r])\setminus \overline{T(x)T(y)}$ is completely on one side of $\overline{T(x)T(y)}$, (c) \& (d) $\overline{xy}\subset \overline{pq}\subset\ell_1$. }
	\end{figure}
	A touching point $T(x)$ of a monotone line in $<_K$ with $T([\mathbf{0}, r])$ can be viewed as a degenerated non-extendable line segment. So $x\in (\ell_1\cap\ell_2) = \{\mathbf{0}, r\}$.
\end{proof}

With the help of Lemma \ref{lem5.1} we are now able to prove Proposition \ref{pro3.1}.
\begin{proof}[Proof of Proposition \ref{pro3.1}]
	For any distinct points $T(x), T(y)\in T([\mathbf{0}, r])$ with $T(x) <_K T(y)$, we show that $\overline{T(x)T(y)} \subset T([\mathbf{0}, r])$. Suppose this is not true. Then there is a point $w\in (\overline{T(x)T(y)} \setminus T([\mathbf{0}, r]))$ and $T(x) <_K w <_K T(y)$. Then there are $T(u), T(v), T(p), T(q)\in T([\mathbf{0}, r])$ such that these points and $T(x), T(y), w$ are on the same straight line satisfying 
	\[
	T(u) \leq_K T(x) \leq_K T(v) <_K w <_K T(p) \leq_K T(y) \leq_K T(q)
	\]
	and both $\overline{T(u)T(v)}$ and $\overline{T(p)T(q)}$ are non-extendable line segments (including the possibility of degenerated line segment consisting of a single point) in $T([\mathbf{0}, r])$. By Lemma \ref{lem5.1}, $u, p\in\ell_2$ and $v, q\in\ell_1$ with $u\leq_K v$ and $p\leq_K q$. Then $T^{-1}(\overline{T(u)T(v)})$ is a monotone curve in $<_K$ from $u$ to $v$ (including a degenerated case of a single point $u=v$) and $T^{-1}(\overline{T(p)T(q)})$ is a monotone curve in $<_K$ from $p$ to $q$ (including a degenerated case of a single point $p=q$). As $\overline{T(u)T(v)}\cap\overline{T(p)T(q)} =\emptyset$ and $T$ is a homeomorphism, we have $T^{-1}(\overline{T(u)T(v)})\cap T^{-1}(\overline{T(p)T(q)}) =\emptyset$. Then these two monotone curves together with the section $\ell_4$ of $\ell_2$ between $u$ and $p$ and the section $\ell_3$ of $\ell_1$ between $v$ and $q$ form the boundary of a region $R\subset [\mathbf{0}, r]$ (see Figure 3). Then $T(R)\subset T([\mathbf{0}, r])$ is a region bounded by $\overline{T(u)T(v)}$, $T(\ell_3)$, $\overline{T(p)T(q)}$ and $T(\ell_4)$. Now take a point $a\in (\ell_4\setminus\{u, p\})$ close to $u$, so $T(a) \in (T(\ell_4)\setminus\{T(u), T(p)\})$ close to $T(u)$. As $T(u), T(v), T(p), T(q)$ are on the same line but $T(a)$ is not, the line passing through $T(a)$ parallel to $\overline{T(x)T(y)}$ must intersect $T(\ell_4)$ at $T(b)>_K T(a)$ for some $b\in\ell_4$ between $a$ and $p$ such that $\overline{T(a)T(b)}$ is a non-extendable line segment in $T([\mathbf{0}, r])$. By Lemma \ref{lem5.1}, $a <_K b$ with $a\in \ell_2$ and $b\in\ell_1$. But $a, b$ are between $u$ and $p$ on $\ell_4$, so $a, b\in (\ell_2\setminus\{\mathbf{0}, r\})$ and $b\not\in \ell_1$. This contradiction shows that $\overline{T(x)T(y)} \subset T([\mathbf{0}, r])$. Hence, $T([\mathbf{0}, r])$ is type-K convex.
	\begin{figure}[htb]
		\includegraphics[width=5.5in,height=3.5in]{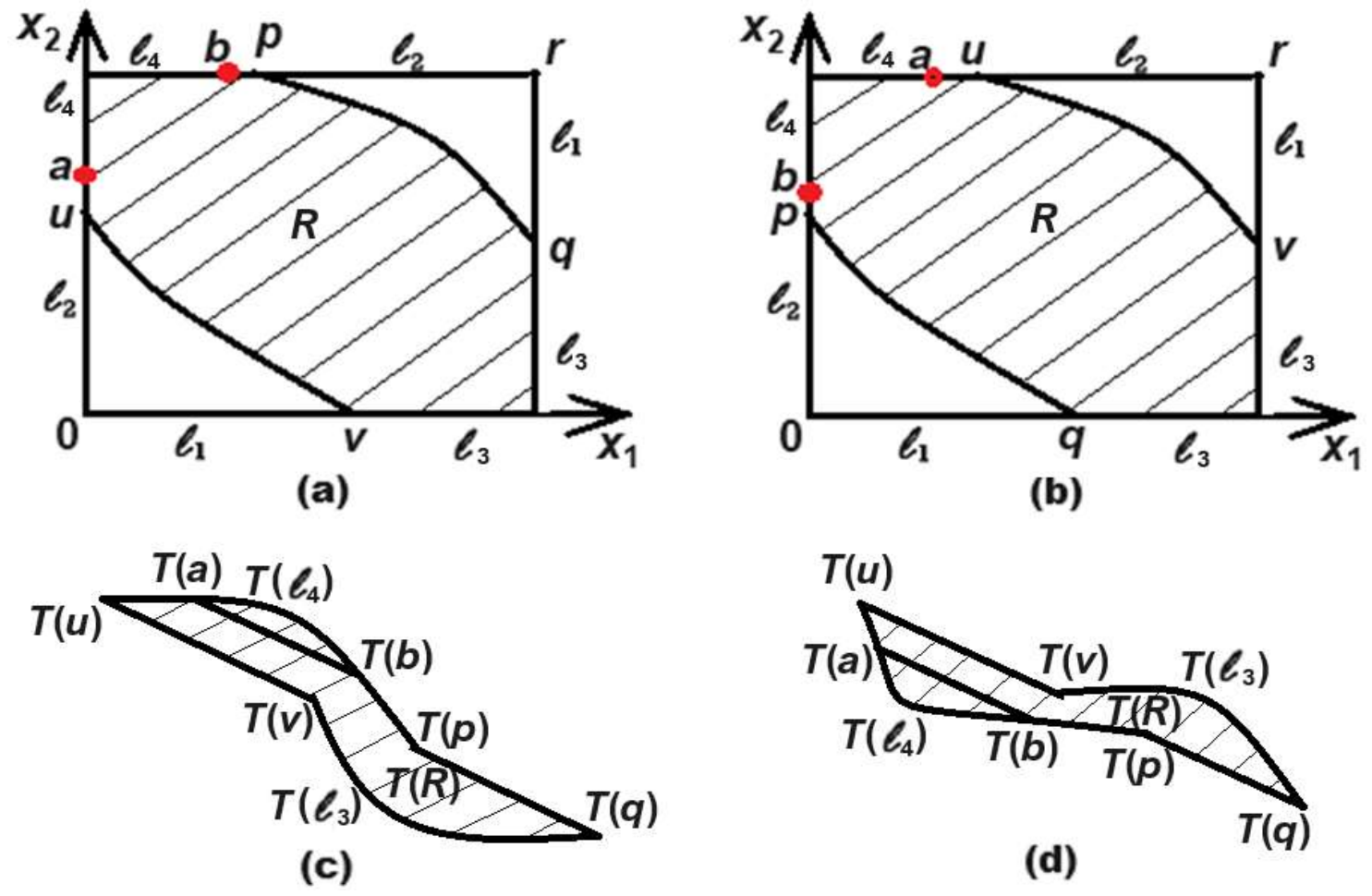}
		\caption{Illustration of $\overline{uv}$, $\overline{pq}$, $\overline{T(u)T(v)}$ and $\overline{T(p)T(q)}$: (a) \& (b) The region $R$ bounded by $T^{-1}(\overline{T(u)T(v)}), T^{-1}(\overline{T(p)T(q)}), \ell_3$ and $\ell_4$. (c) \& (d) The region $T(R)$ bounded by $\overline{uv}, T(\ell_3), \overline{pq}$ and $T(\ell_4)$. }
	\end{figure}
\end{proof}

\begin{proof}[Proof of Theorem \ref{the3.2}]
	(a) From $T([\mathbf{0}, r])\subset [\mathbf{0}, r)$ we see that $T_1(r_1, 0) = r_1f_1(r_1, 0)< r_1$ and $T_2(0, r_2) = r_2f_2(0, r_2) < r_2$, so $f_1(r_1, 0) <1$ and $f_2(0, r_2) <1$. By (A2) $f(\mathbf{0})\gg \mathbf{1}$. Since $f_1(t, 0), f_2(0, t)$ are decreasing in $t$ by (A1), there is a unique point $Q_i$ on the positive half $x_i$-axis such that $f_i(Q_i) =1$, i.e. $T(Q_i) = Q_i$, for $i = 1, 2$. Let $Q_1 =(q_1, 0)$ and $Q_2 =(0, q_2)$.
	
	From (\ref{e6}) and (\ref{e7}) we see that $DT(Q_1)$ has eigenvalues $1+q_1\frac{\partial f_1}{\partial x_1}(Q_1)$ and $f_2(Q_1)$. As $\rho(M(x)) <1$, from (\ref{e12}) $-q_1\frac{\partial f_1}{\partial x_1}(Q_1)< 2$ so $|1+q_1\frac{\partial f_1}{\partial x_1}(Q_1)| <1$. As $f_2(s, 0)$ is nondecreasing in $s$ and $f(\mathbf{0})\gg\mathbf{1}$, we have $f_2(Q_1)>1$. This shows that $Q_1$ is a saddle point. Clearly, $W^s(Q_1)$ is the positive half $x_1$-axis and $W^u(Q_1)\subset\dot{C}$. Similarly, $Q_2$ is a saddle point, $W^s(Q_2)$ is the positive half $x_2$-axis and $W^u(Q_2)\subset \dot{C}$.
	
	(b) Now consider the sets defined by
	\begin{equation} \label{e13}
		\ell_1 = \{x\in C: f_1(x) =1\}, \quad \ell_2 = \{x\in C: f_2(x) =1\}.	
	\end{equation}
	Then the assumption (A1) and implicit function theorem imply the existence of $C^1$ positive nondecreasing functions $g_i: \R_+\to\R_+$ such that $\ell_i$ is the graph of $g_i$ for $i=1, 2$,  $x_1 = g_1(x_2)$ and $x_2 = g_2(x_1)$ satisfy $f_1(g_1(x_2), x_2)\equiv 1$, $f_2(x_1, g_2(x_1))\equiv 1$, and 
	\[
	g_1'(x_2) = - \frac{\partial f_1}{\partial x_2}/\frac{\partial f_1}{\partial x_1} \geq 0,\quad g_2'(x_1) = - \frac{\partial f_2}{\partial x_1}/\frac{\partial f_2}{\partial x_2} \geq 0.
	\]
	The curve $\ell_1$ ($\ell_2$) divides $C$ into two regions: one containing the positive half $x_2$-axis ($x_1$-axis) with $f_1(x)>1$ ($f_2(x)>1$) and the other with $f_1(x) <1$ ($f_2(x) <1$). By the assumptions, we need only consider the region $[\mathbf{0}, r]$ and $\ell_1$ and $\ell_2$ divide it into at least four regions: $R_1$ containing $\mathbf{0}$ with $f(x)\gg \mathbf{1}$, $R_2$ containing $r$ with $f(x)\ll\mathbf{1}$, $R_3$ containing part of $x_1$-axis with $f(x)\ll_K\mathbf{1}$, $R_4$ containing part of $x_2$-axis with $f(x)\gg_K\mathbf{1}$, and possibly more. Then we see the existence of intersection points $p_0, p_1\in (\mathbf{0}, r)$ of $\ell_1$ and $\ell_2$, $p_0$ on the common boundary of $R_1, R_3, R_4$ and $p_1$ on the common boundary of $R_2, R_3, R_4$. The intersection points of $\ell_1$ and $\ell_2$ are the fixed points of $T$ in $(\mathbf{0}, r)$. Thus, $T$ has at least one fixed point in $(\mathbf{0}, r)$. If $p_0=p_1$ then $p_0$ is the unique fixed point in $(\mathbf{0}, r)$ and $R_i, 1\leq i\leq 4$, are the only four regions. If $p_0\not= p_1$ and the sections of $\ell_1$ and $\ell_2$ between $p_0$ and $p_1$ are not completely coincide with each other, then $T$ may have other fixed points between $p_0$ and $p_1$. By the nondecreasing nature of $g_1$ and $g_2$, any two distinct fixed points in $(\mathbf{0}, r)$ are related by $\ll$ and $p_0\ll p\ll p_1$ for any other fixed point $p\in (\mathbf{0}, r)$. In this case each $R$ of the other regions is bounded by $\ell_1$ and $\ell_2$ with two fixed points on the boundary and $f(x)\ll \mathbf{1}$ in the region or $f(x)\gg \mathbf{1}$ in the region (see Figure 4).
	\begin{figure}[htb]
		\includegraphics[width=5.5in,height=3in]{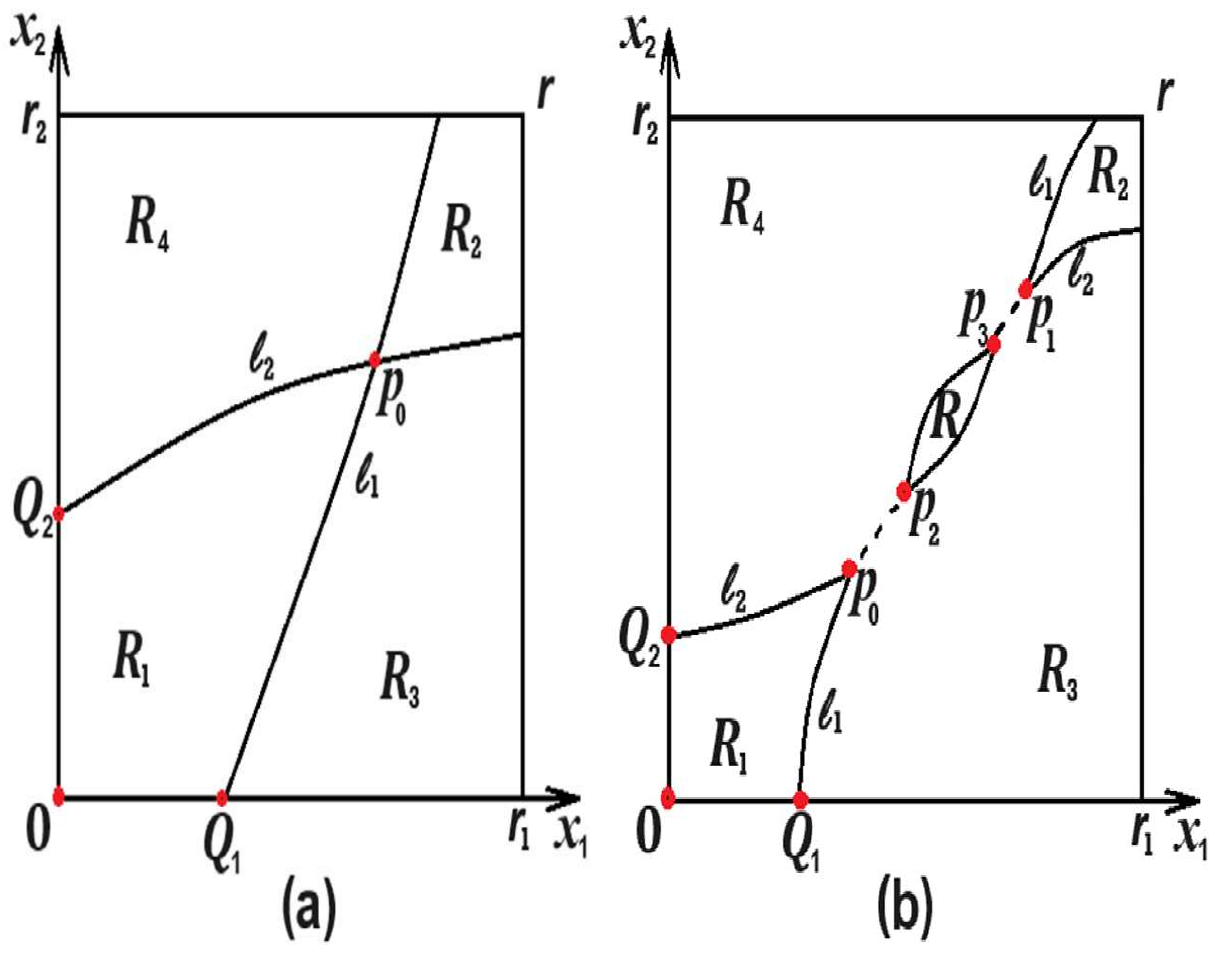}
		\caption{Illustration of the regions divided by $\ell_1$ and $\ell_2$: (a) Only four regions $R_1$--$R_4$, (b) Five or more regions. }
	\end{figure}
	
	We next prove that each of $\overline{R_1}, \overline{R_2}$ and $\overline{R}$, if any, is forward invariant. For any $x\in (\overline{R_1}\setminus\{p_0\})\cap\dot{C}$ we have $T(x) > x$. If $T(x)\in R_3$ then $T^2(x) \ll_K T(x)$. From Proposition \ref{pro3.1} and Remark 1 we know that $T$ is weakly type-K retrotone. Thus, $T(x) = T^{-1}(T^2(x)) \ll_K T^{-1}(T(x)) = x$, a contradiction to $T(x) > x$. If $T(x)\in R_4$ then $T^2(x) \gg_K T(x)$. As $T$ is weakly type-K retrotone, $T(x) = T^{-1}(T^2(x)) \gg_K T^{-1}(T(x)) = x$, a contradiction to $T(x) > x$. If $T(x)$ is in $\overline{R_2}$ or some $\overline{R} $, then the homeomorphism $T$ maps the line segment $\overline{\mathbf{0}x} \subset (\overline{R_1}\setminus\{p_0\})$ onto a curve from $\mathbf{0}$ to $T(x)$. Since this curve can pass through neither $R_3$ nor $R_4$, it must pass through $p_0$, a contradiction to $T^{-1}(p_0) = p_0\not\in \overline{\mathbf{0}x}$. All these contradictions show that $T(x)\in\overline{R_1}$. If $x\in R_1$ is on the positive half $x_1$-axis then $T(x)$ is on the same half axis. As $T$ is a homeomorphism, we must have $T(x)<Q_1$, so $T(x)\in R_1$. Similarly, if $x\in R_1$ is on the positive half $x_2$-axis, then so is $T(x)$ and $T(x)\in R_1$. Therefore, $T(\overline{R_1}) \subset \overline{R_1}$. 
	
	For any $x\in (\overline{R_2}\setminus\{p_1\})$ we have $T(x)< x$. By the same reasoning as above, $T(x)\not\in (R_3\cup R_4)$. For the point $(a, r_2)\in \ell_1\cap ([\mathbf{0}, r]\setminus [\mathbf{0}, r))$, it is in $\overline{R_2}$ so $T(a, r_2)< (a, r_2)$ and $T(a, r_2)\not\in(R_3\cup R_4)$. If $T(a, r_2)\not\in\overline{R_2}$ then it must satisfy $T(a, r_2)< p_1=T(p_1)<(a, r_2)$. But $(a, r_2)\in \ell_1$ implies $T_1(a, r_2) = a$. Then combination of $T(a, r_2)< p_1=T(p_1)<(a, r_2)$ and $a\leq T_1(p_1)\leq a$ gives $T_1(a, r_2) = T_1(p_1) = a$ and $T_2(a, r_2) < T_2(p_1) <r_2$. Thus, $(a, r_2) <_K p_1 = T(p_1) <_K T(a, r_2)$. By weakly type-K retrotone property of $T$, $p_1 <_K(a, r_2)$, a contradiction to $(a, r_2) <_K p_1$. This shows that $T(a, r_2)\in\overline{R_2}$. If $T(x)\not\in\overline{R_2}$ for some $x\in(\overline{R2}\setminus\{p_1\})$, then $T(x)<p_1$. Now taking a simple curve $L\subset (\overline{R_2}\setminus\{p_1\})$ from $(a, r_2)$ to $x$, $T(L)$ is a simple curve from $T(a, r_2)\in\overline{R_2}$ to $T(x)<p_1$. As the curve can pass through neither $R_3$ nor $R_4$, it must pass through $p_1$, a contradiction to $T^{-1}(p_1) = p_1\not\in L$. Therefore, $T(\overline{R_2})\subset \overline{R_2}$.
	
	Suppose a region $R$ exists with two fixed points $p_2$ and $p_3\gg p_2$ on its boundary. Then $T(x) <x$ for all $x\in(\overline{R}\setminus\{p_2, p_3\})$ ($T(x) >x$ for all $x\in(\overline{R}\setminus\{p_2, p_3\})$), and $T(x)\not \in(R_3\cup R_4)$ by the same reasoning as that used in $\overline{R_1}$ and $\overline{R_2}$. If $T(x)\not\in\overline{R}$ for some $x\in (\overline{R}\setminus\{p_2, p_3\})$, then $T(x)<p_2$ ($T(x)>p_3$). Choosing a simple curve $L\subset (\overline{R}\setminus\{p_2, p_3\})$ from $p_3$ ($p_2$) to $x$, $T(L)$ is a simple curve from $p_3$ ($p_2$) to $T(x)$ passing through $p_2$ ($p_3$), a contradiction to $T^{-1}(p_2) = p_2\not\in L$ ($T^{-1}(p_3) = p_3\not\in L$). Hence, $T(\overline{R}) \subset \overline{R}$.
	
	Since $T(x)\ll_K x$ for all $x\in (R_3\cap\dot{C})$ and $T(x)\gg_K x$ for all $x\in (R_4\cap\dot{C})$, by the weakly type-K retrotone property of $T$, $T(R_3)\cap R_4 =\emptyset$ and $T(R_4)\cap R_3 =\emptyset$. From this it follows that, for each $x\in (R_3\cup R_4)\cap\dot{C}$, either $T^n(x)$ is in $(\overline{R_1}\cup\overline{R_2})$ or some $\overline{R}$ for some $n>0$ or it stays in $(R_3\cup R_4)\cap\dot{C}$ for all $n\geq 0$.
	
	In summary, we conclude that for each $x\in\dot{C}$, the sequence $\{T^n(x)\}$ is eventually monotone in $\ll_K$ or $<$ so it converges to a fixed point. If $p_0$ is the unique fixed point in $\dot{C}$ then $\lim_{n\to\infty}T^n(x) = p_0$ for all $x\in\dot{C}$. The stability of $p_0$ follows from the dynamics on the four regions $R_1$--$R_4$.
	
	(c) Since $\mathcal{R}(\mathbf{0})$ is open in $C$, from Proposition \ref{pro2.4} (a) we have $\overline{\mathcal{R}(\mathbf{0})}\setminus \mathcal{R}(\mathbf{0})\subset \Sigma$. Consider the line segment $\ell(s)=\{q_1^{-1}x_1+q_2^{-1}x_2 = s: x\in C\}$ for each $s>0$. Then each pair of distinct points on $\ell(s)$ are related by $\ll_K$ and $\ell(1) = \overline{Q_1Q_2}$. From (A1), part (b) above and Proposition \ref{pro2.4} (b) we see the existence of $s_0>1$ such that 
	\[
	\forall s\geq s_0, \ell(s)\cap \mathcal{R}(\mathbf{0}) = \emptyset; \quad \forall s\in(0, s_0), \ell(s)\cap \mathcal{R}(\mathbf{0}) \not= \emptyset.
	\] 
	Moreover, for each $s\in [1, s_0)$, there are two distinct points $A(s), B(s)\in\overline{\mathcal{R}(\mathbf{0})}$ with $A(s)\ll_K B(s)$ such that 
	\[
	\overline{A(s)B(s)}\setminus\{A(s), B(s)\} = \ell(s)\cap \mathcal{R}(\mathbf{0}).
	\] 
	By Proposition \ref{pro2.4} (c), $A(s)\in \Sigma_V$ and $B(s)\in \Sigma_H$ for all $s\in [1, s_0)$. We first show that $L_H=\{B(s): 1\leq s < s_0\}$ and $L_V = \{A(s): 1\leq s < s_0\}$ are monotone curves in $<$. Note that any two distinct points in $\R^2$ are related by $\ll_K$ or $<$. For any $1\leq s_1<s_2<s_0$, the line segment $\ell(s_1)$ is below $\ell(s_2)$ so we cannot have $B(s_2)< B(s_1)$. By Proposition \ref{pro2.4} (c) again, $B(s_1)$ and $B(s_2)$ are not related by $\ll_K$ so we must have $B(s_1)<B(s_2)$. By the same reasoning as above, we also have $A(s_1)<A(s_2)$. Thus, $L_H$ and $L_V$ are monotone. For any $\bar{s}\in [1, s_0)$, since the line segment $\ell(s)$ is parallel to $\ell(\bar{s})$, as $s\to \bar{s}+$ the endpoint $A(s)$ ($B(s)$) converges to a point $A'\in \ell(\bar{s})$ ($B'\in \ell(\bar{s})$). Thus, $A'$ and $A(\bar{s})$ ($B'$ and $B(\bar{s})$) can only be related by $\ll_K$ or $=$. But since $A(\bar{s})<A(s)$ ($B(\bar{s})<B(s)$) for all $s\in(\bar{s}, s_0)$, we should have $A(\bar{s}) \leq A'$ ($B(\bar{s}) \leq B'$), which rules out the possibility of $A'$ and $A(\bar{s})$ ($B'$ and $B(\bar{s})$) related by $\ll_K$. Thus, $A' = A(\bar{s})$ ($B' = B(\bar{s})$). By the same reasoning, we also have $\lim_{s\to \bar{s}-}A(s) = A(\bar{s})$ ($\lim_{s\to \bar{s}-}B(s) = B(\bar{s})$) if $\bar{s}>1$. Therefore, $A(s)$ and $B(s)$ are continuous, so $L_H$ and $L_V$ are monotone curves in $<$.
	
	Now let $\tilde{A} = \lim_{s\to s_0-}A(s)$ and $\tilde{B} = \lim_{s\to s_0-}B(s)$. Then $\tilde{A}, \tilde{B}\in\ell(s_0)$. It follows from this and $A(s)\ll_K B(s)$ that either $\tilde{A}\ll_K\tilde{B}$ or $\tilde{A}=\tilde{B}$. Suppose $\tilde{A}\ll_K\tilde{B}$. Since $\overline{\tilde{A}\tilde{B}}$ is on the north-east boundary of $\mathcal{R}(\mathbf{0})$, for any point $x\in (\overline{\tilde{A}\tilde{B}}\setminus\{\tilde{A}, \tilde{B}\})$, the horizontal and vertical lines through $x$ both intersect with $\mathcal{R}(\mathbf{0})$. So there are $y, z\in\mathcal{R}(\mathbf{0})$ satisfying $y<_K x<_K z$. Then the weakly type-K retrotone property of $T$ implies 
	\[
	0 = \lim_{n\to\infty}T^{-n}(y) \leq_K \lim_{n\to\infty}T^{-n}(x) \leq_K \lim_{n\to\infty}T^{-n}(z) =0,
	\]
	so $x\in\mathcal{R}(\mathbf{0})$, a contradiction to $\ell(s_0)\cap \mathcal{R}(\mathbf{0}) = \emptyset$. This shows that $\tilde{A}=\tilde{B}$. Then $\overline{\mathcal{R}(\mathbf{0})} = L_H\cup L_V\cup\{\tilde{A}\}\cup \mathcal{R}(\mathbf{0})$. From the proof of part (b) above we see that the region $R_1$ is forward invariant containing $\overline{\mathcal{R}(\mathbf{0})}$. Since $\lim_{n\to\infty}T^n(x) = p_0$ for all $x\in (R_1\cap\dot{C})$, we must have $p_0\in (L_H\cup L_V\cup\{\tilde{A}\})$. But as $p_0\gg \mathbf{0}$, $p_0\not\in(L_H\cup L_V)\subset (\Sigma_H\cup\Sigma_V)$. Thus, $\tilde{A} = p_0\in\Sigma_0$. 
	
	Note that $\Sigma$ consists of bounded whole orbits outside of $\mathcal{R}(\mathbf{0})$. For each $x\in [\mathbf{0}, r]\setminus ([p_0, p_1]\cup \overline{\mathcal{R}(\mathbf{0})})$, $T^{-n}(x)$ either does not exist or is outside of $[\mathbf{0}, r]$ for some integer $n>0$. Thus, $\Sigma \cap ([\mathbf{0}, r]\setminus ([p_0, p_1]\cup \overline{\mathcal{R}(\mathbf{0})}))= \emptyset$ and $\Sigma \subset ([p_0, p_1]\cup \overline{\mathcal{R}(\mathbf{0})}\setminus \mathcal{R}(\mathbf{0}))$. 
	
	If $p_0=p_1$ then 
	\[
	\Sigma = \overline{\mathcal{R}(\mathbf{0})}\setminus \mathcal{R}(\mathbf{0}),\quad \Sigma_H = L_H, \quad \Sigma_V = L_V, \quad\Sigma_0 = \{p_0\}.
	\]
	If $p_0\ll p_1$, from the analysis in (b) above we see that $T([p_0, p_1])\subset [p_0, p_1]$ and
	\[
	\cdots \subset T^{n+1}([p_0, p_1])\subset T^n([p_0, p_1])\subset \cdots \subset [p_0, p_1]. 
	\] 
	As $[p_0, p_1]$ is compact and $T$ maps compact set to compact set, $\cap_{n=0}^{\infty}T^n([p_0, p_1])$ is nonempty compact and invariant. Thus, $\cap_{n=0}^{\infty}T^n([p_0, p_1])\subset \Sigma_0$. But since $\cap_{n=0}^{\infty}T^n([p_0, p_1])$ is the largest invariant subset in $[p_0, p_1]$ and $\Sigma_0$ is an invariant subset of $[p_0, p_1]$, we must have $\Sigma_0\subset \cap_{n=0}^{\infty}T^n([p_0, p_1])$. Hence, $\cap_{n=0}^{\infty}T^n([p_0, p_1]) = \Sigma_0$, $L_H=\Sigma_H$ and $L_V=\Sigma_V$. 
	
	We now show that $\Sigma_0$ is a monotone curve from $p_0$ to $p_1$. Indeed, let $L_1$ ($L_2$) be the upper and left (lower and right) sides of the rectangle $[p_0, p_1]$. As $T$ is a homeomorphism, for each $n\geq 0$, $T^n(L_1), T^n(L_2)$ are simple curves from $p_0$ to $p_1$ as the boundary of $T^n([p_0, p_1])$. There is $s_1>s_0$ such that $\ell(s_1)\cap (\cap_{n=0}^{\infty}T^n([p_0, p_1])) =\{p_1\}$ and, for each $s\in (s_0, s_1)$, there are points $a_n(s)\in T^n(L_1), b_n(s)\in T^n(L_2)$ such that $\overline{a_n(s)b_n(s)} \subset T^n([p_0, p_1])$. These line segments on $\ell(s)$ are compact and 
	\[
	\cdots \subset \overline{a_{n+1}(s)b_{n+1}(s)}\subset \overline{a_n(s)b_n(s)}\subset \cdots \subset \overline{a_0(s)b_0(s)}. 
	\]
	Thus, $\cap_{n=0}^{\infty}\overline{a_n(s)b_n(s)}$ is a nonempty compact subset of $\Sigma_0$. If there are two distinct points in $\cap_{n=0}^{\infty}\overline{a_n(s)b_n(s)}$, by Proposition \ref{pro2.4} (c) one would be in $\Sigma_H$ and the other in $\Sigma_V$. This is impossible as these points are in $\Sigma_0$. Hence, there is a point $c(s)\in\Sigma_0$ such that $\cap_{n=0}^{\infty}\overline{a_n(s)b_n(s)}= \{c(s)\}$, i.e. each line $\ell(s)$ intersects $\Sigma_0$ at $c(s)$. For any $s_0\leq s<s'\leq s_1$, since $\ell(s)$ is below $\ell(s')$, we cannot have $c(s')\leq c(s)$. If either $c(s)<_Kc(s')$ or $c(s')<_Kc(s)$, from Proposition \ref{pro2.4} (c) we obtain either $\lim_{n\to\infty}[T^{-n}(c(s))]_1 = 0$ or $\lim_{n\to\infty}[T^{-n}(c(s))]_2 = 0$, a contradiction to the invariance of $\Sigma_0\subset [p_0, p_1]$. Therefore, $c(s)\ll c(s')$. The continuity of $c(s)$ follows from the same reasoning as for the continuity of $A(s)$ and $B(s)$ above. This shows that $\Sigma_0$ is a monotone curve in $\ll$ from $p_0$ to $p_1$. 
	
	It is now clear from above that $\Sigma_H = W^u(Q_1)\cup\{Q_1\}$ and $\Sigma_V = W^u(Q_2)\cup\{Q_2\}$. 
	
	Finally, if $\frac{\partial f_1(x)}{\partial x_2} >0$ and $\frac{\partial f_2(x)}{\partial x_1} >0$ for $x\in [\mathbf{0}, r]$, by Proposition \ref{pro3.1} and Proposition \ref{pro2.3} $T$ is type-K retrotone. Then $\ell(s_1)$ below $\ell(s_2)$ implies $B(s_2)\not\ll B(s_1)$ and, by Proposition \ref{pro2.4} (d), $B(s_1)$ and $B(s_2)$ are not related by $<_K$. So we must have $B(s_1)\ll B(s_2)$. Hence, $L_H$ is monotone in $\ll$. Similarly, $H_V$ is monotone in $\ll$.   
\end{proof}

\section{Conclusion}\label{Sec6}
So far we have considered the discrete dynamical system (\ref{e1}) with the map $T$ defined by (\ref{e2}). We have introduced the concepts of a type-K competitive system, a type-K retrotone (weakly type-K retrotone) map and a type-K convex set. We note that (weakly) type-K retrotone maps are backward monotone in the order $<_K$. We then have established that a general $N$-dimensional type-K competitive system under certain conditions means the map $T: [\mathbf{0}, r]\to T([\mathbf{0}, r])$ is a homeomorphism and (weakly) type-K retrotone. Under these conditions and the assumptions that (\ref{e1}) with (\ref{e2}) is type-K competitive, that the origin $\mathbf{0}$ is a repeller, that every orbit $T^n(x)$ enters into the forward invariant set $[\mathbf{0}, r]$, we have shown the existence of the global attractor $\Sigma$ of the system on $C\setminus\{\mathbf{0}\}$ and its decomposition $\Sigma = \Sigma_H\cup\Sigma_0\cup\Sigma_V$ with $\Sigma_0= \Sigma\setminus(\Sigma_h\cup\Sigma_V)$ and (\ref{e10}). 

We have further proved that $\Sigma_H\cap C_H$ is the global attractor of the system restricted to $C_H\setminus\{\mathbf{0}\}$, the (modified) carrying simplex of a $k$-dimensional competitive subsystem, that $\Sigma_V\cap C_V$ is the global attractor of the system restricted to $C_V\setminus\{\mathbf{0}\}$, the (modified) carrying simplex of an $(N-k)$-dimensional competitive subsystem, that both $\Sigma_H\cap C_H$ and $\Sigma_V\cap C_V$ repel on $\Sigma$, and that $\Sigma_0$ is the global attractor of the system restricted to $C\setminus(C_H\cup C_V)$. This supports the following simplified view: viewing each of $C_H$ and $C_V$ as ``one-dimensional'' positive half axis and $C$ as the first quadrant in the ``two-dimensional'' plane, $\Sigma_H\cap C_H$ can be viewed as a saddle point on $C_H$ with $C_H$ as its stable manifold and $\Sigma_H$ as its unstable manifold, $\Sigma_V\cap C_V$ can be viewed as a saddle point on $C_V$ with $C_V$ as its stable manifold and $\Sigma_V$ as its unstable manifold, and $\Sigma_0$ can be viewed as a global attractor inside the first quadrant.

We have also proved that $\Sigma_H\cup\Sigma_0$ and $\Sigma_V\cup\Sigma_0$ are unordered in $\ll_K$. This means that each of $\Sigma_H\cup\Sigma_0$ and $\Sigma_V\cup\Sigma_0$ is one-one correspondent to a bounded subset of an $(N-1)$-dimensional plane. These enable us to point out the remaining problems to be investigated (see the Open Problems below).

For planar type-K competitive systems, we have established a clear picture about geometric and dynamic features on $\Sigma$: both $\Sigma_H\cup\Sigma_0$ and $\Sigma_V\cup\Sigma_0$ are one-dimensional monotone curves in $<$ and compact invariant manifolds. The dynamics of such planar systems is simple: every forward orbit is eventually monotone either in $<$ or $\ll_K$, thus every forward orbit converges to a fixed point.

We now end this paper with the following:
\vskip 3 mm
\noindent \textbf{Open Problems} Based on the results of Propositions \ref{pro2.4} and \ref{pro2.5}, suppose $N\geq 3$ for (\ref{e1}) with (\ref{e2}).
\begin{itemize}
	\item[1)] Find the geometric features of $\Sigma$. Is $\Sigma_H$, $\Sigma_V$ or $\Sigma$ an $(N-1)$-dimensional surface?
	\item[2)] Find the dynamical feature of the system on $\Sigma$. Can the dynamics of the system on $C\setminus\{\mathbf{0}\}$ be reprsented by the dynamics restricted to $\Sigma$? In other words, is every nontrivial forward orbit asymptotic to one in $\Sigma$?
	\item[3)] Note that the conditions $(a)$-$(c)$ of Proposition \ref{pro2.3} ensure that $T: [\mathbf{0}, r]\to T([\mathbf{0}, r])$ is a homeomorphism and weakly type-K retrotone. Condition (b) is difficult to check for any concrete system. Tempted by the planar case, can $(a)$ and (c) imply (b)? Or is there an easily checkable sufficient condition for (b) to hold? 
	\item[4)] Tempted by the planar case, for $N=3$, can we find a detailed configuration or classification of $\Sigma$ for complete understanding of 3D type-K competitive systems?
\end{itemize}


\medskip
\medskip

\end{document}